\documentclass[sn-mathphys,Numbered]{sn-jnl}

\usepackage{graphicx}%
\usepackage{multirow}%
\usepackage{amsmath,amssymb,amsfonts}%
\usepackage{amsthm}%
\usepackage{mathrsfs}%
\usepackage[title]{appendix}%
\usepackage{xcolor}%
\usepackage{textcomp}%
\usepackage{manyfoot}%
\usepackage{booktabs}%
\usepackage{algorithm}%
\usepackage{algorithmicx}%
\usepackage{algpseudocode}%
\usepackage{listings}%

\jyear{2023}%

\theoremstyle{thmstyleone}%
\newtheorem{theorem}{Theorem}
\newtheorem*{conjecture}{Conjecture}
\newtheorem{lemma}[theorem]{Lemma}

\theoremstyle{thmstyletwo}%
\newtheorem{remark}[theorem]{Remark}%

\theoremstyle{thmstylethree}%
\newtheorem{definition}[theorem]{Definition}%
\begin{document}
\allowdisplaybreaks[3]
\title[Proof of a conjecture for two-armed slot machines]{Proof of a conjecture about Parrondo's paradox for two-armed slot machines}

\author*[1]{\fnm{Huaijin} \sur{Liang}}\email{652676760@qq.com}
\author[2]{\fnm{Zengjing} \sur{Chen}}\email{zjchen@sdu.edu.cn}

\equalcont{These authors contributed equally to this work.}

\affil*[1]{\orgdiv{Zhongtai Securities Institute for Financial Studies}, \orgname{Shandong University}, \orgaddress{\street{No. 27, Shanda South Road}, \city{Jinan City}, \postcode{250100}, \state{Shandong Province}, \country{China}}}

\affil[2]{\orgdiv{Zhongtai Securities Institute for Financial Studies}, \orgname{Shandong University}, \orgaddress{\street{No. 27, Shanda South Road}, \city{Jinan City}, \postcode{250100}, \state{Shandong Province}, \country{China}}}


\abstract{The 1936 Mills Futurity slot machine had the feature that, if a player loses 10 times in a row, the 10 lost coins are returned. Ethier and Lee (2010) studied a generalized version of this machine, with 10 replaced by deterministic parameter $J$. They established the Parrondo eﬀect for a hypothetical two-armed machine with the Futurity award. Speciﬁcally, arm $A$ and arm $B$, played individually, are asymptotically fair, but when alternated randomly (the so-called random mixture strategy), the casino makes money in the long run. They also considered the nonrandom periodic pattern strategy for patterns with $r$ $A$s and $s$ $B$s (e.g., $ABABB$ if $r = 2$ and $s = 3$). They established the Parrondo eﬀect if $r + s$ divides $J$, and conjectured it in four other situations, including the case $J = 2$ with $r \ge 1$ and $s \ge 1$. We prove the conjecture in the latter case.}




\maketitle

\noindent
\section{Introduction}\label{sec1}

The age of origin of human gambling is estimated to be about the same as that of human civilization. All casino games, such as blackjack, slot machines, roulette, and baccarat, no matter how they are changed, conceal the secret phenomenon of ``long-term bettors will lose": no matter how much any given player occasionally wins, any player who continues to gamble will eventually return winnings to the casino. The law of inevitability behind casinos is determined by the ``law of large numbers" in probability theory. Before the establishment of probability theory and its application to gambling, the secret behind casino profitability remained an empirical phenomenon that each person simply learned from experience. Only with the discovery of the mathematical law of large numbers did it become possible to concretely understand casino profitability. The profitability originates from the ``long-term bettors will lose" phenomenon, which occurs whenever a gambler's asymptotic expected profit per coup is negative, compared to a truly fair game, in which both sides have an asymptotic expected profit per coup of zero. For a game of independent coups, it is
straightforward to determine fairness, but when coups are dependent on past results this determination becomes more challenging.

The Futurity slot machine, designed by the Mills Novelty Company of Chicago, was in production from 1936 to 1941. This slot machine features a pointer at the top of the machine to indicate the current number of consecutive player
losses. It advances by 1 after each loss and resets to 0 after each win. When the pointer reaches 10, all 10 lost coins are returned to the player, and the pointer resets to 0. This payoff is called the Futurity award. Furthermore, the machine's payout distribution depends on the mode it is in, and the mode varies deterministically, having a period of 10. The mode is determined by a cam that rotates through 10 positions. Each time the player pulls the arm, the cam rotates to the next position. Geddes and Saul \cite{geddes1980mills,geddes1980mathematics} used Monte Carlo simulation to study the Futurity, and Ethier and Lee \cite{ethier2010markovian}  later derived analogous
results using analytical methods. In the latter's formulation, for the number of
consecutive losses needed for the Futurity award, 10 was replaced by $J \ge 2$,  and
for the number of cam positions used to determine the mode, 10 was replaced
by $I$, assumed to be a multiple of $J$ (i.e., $I = dJ$ for a positive integer $d$). Ethier and Lee \cite{ethier2010markovian} established a strong law of large numbers for the generalized
Futurity slot machine. The special case $I = J = 10$ corresponds to the original machine.

Parrondo's paradox \cite{parrond1996cheat}  is the counterintuitive conclusion that there exist
two fair games that can be combined, by either random mixture or nonrandom
alternation, to create an unfair game. It was motivated by a model of physical
transport based on Brownian ratchets. See Harmer and Abbott \cite{harmer1999parrondo}. Examples
of Parrondo's paradox and its generalizations \cite{abbott2010asymmetry} can be found in fields such as physics, biology, and economics. Parrondo's paradox also has applications in information theory \cite{harmer2000information}, quantum game theory \cite{abbott2002order}, and other areas \cite{harmer2000paradox,pyke2003random,ethier2009limit,ethier2019strong}.

This paper proceeds as follows. In Section \ref{sec2}, we state some of the results already proven by Ethier and Lee in \cite{ethier2010markovian} and restate their conjecture. In Section \ref{sec3}, we present our main results. In Section \ref{sec5}, we establish the specific value of the casino's asymptotic profit per coup under a nonrandom periodic pattern player strategy (such as $D=ABABB$). In Section \ref{sec6}, we use a typical example to demonstrate one of our means of proving the conjecture. We show that no matter what nonrandom periodic pattern strategy the player implements, the casino remains profitable in the long run, verifying the conjecture of Ethier and Lee.

To reduce the length of this paper, the complete details of the proof are
deferred to Appendices A, B, and C.
version of this paper.

\section{Previous conclusions and the conjecture}\label{sec2}

According to Parrondo's paradox, a winning game can be produced by combining two or more fair games. Pyke \cite{pyke2003random} conceptualized Parrondo's paradox as a two-armed slot machine, in which each arm is ``fair" implying that a gambler who plays either arm ensures that the average cost per coup is close to 0 as the number of coups increases. However, the casino does not restrict the gambler to using a single arm. Pyke \cite{pyke2003random} raised the question of whether casinos can earn profit from ``fair" games. The Parrondo effect is observed when the casino can earn a profit under certain circumstances. Ethier and Lee \cite{ethier2010markovian} used the two-armed Futurity slot machine to answer Pyke's question in the affirmative. We first analyze the generalized one-armed version of the Futurity slot machine and its related conjectures.
\subsection*{One-armed Model:} The internal cam of the slot machine comprises two distinct non-Futurity payout modes sequentially arranged according to the cam's position, advancing one position per coup, which is repeated ad infinitum. Assume that the cam of the slot machine internal control pattern has $I$ positions, denoted as $0$, $1$,$\cdots$, $I-1$. Since the cam rolling is periodic, and any integer multiple of $I$ can be considered a new cam position of the original slot machine. In order to make the results of the one-armed machine can be applied to the two-armed machine played with a particular nonrandom periodic pattern, we can regard $I$ as an integer multiple of $J$ and we assume $I=dJ$ for a positive integer $d$. An irreducible period-$I$ Markov chain ${(X_n,Y_n)}_{n\ge0}$ that controls the slot machine has state space $\sum:=\{0,1,\cdots,I-1\}\times\{0,1,\cdots,J-1\}$. State $(i, j)$ at time $n$ means that the position of cam is at $i$, and the pointer to record the number of consecutive failures is at $j$ after the $n$th coup. Then, the Markov chain with transition probabilities
\begin{align*}
P((i,j),(k,l))&:=P((X_{n+1},Y_{n+1})=(k,l)\mid(X_n,Y_n)=(i,j))\\
&=
\begin{cases}
{p}_i& \text{ if $(k,l)=(i+1($mod $I),0)$ and $j\le J-2$ } \\
{q}_i & \text{ if $(k,l)=(i+1($mod $I),j+1)$ and $j\le J-2$ }\\
1 & \text{ if $(k,l)=(i+1($mod $I),0)$ and $j= J-1$ ,}
\end{cases}
\end{align*}
where $({p}_0,{p}_1,\cdots,{p}_{I-1})$ are the winning probabilities in the various cam positions without Futurity award and ${q}_i:=1-{p}_i$.
At equilibrium, let $\pi$ be the unique stationary distribution (without stating the
formula for it explicitly) and $p^\circ$ is the asymptotic probability that the gambler wins the $J$-coin Futurity award at a particular coup as follows:
\begin{align}\label{eq1}
p^\circ=\sum\limits_{i=0}^{I-1}\pi(i,J-1)q_i=\frac{1}{I\bigg(1-\prod\limits_{j=0}^{I-1}{q}_j\bigg)}\sum\limits_{i=0}^{I-1}\sum\limits_{k=1}^{d}{p}_{i-kJ}\prod\limits_{l=i-kJ+1}^{i}{q}_l,
\end{align}
where ${p}_{i-mI}:={p}_i$ for all $i=0,1,\cdots,I-1$ and $m\ge1$.

Now, we return to the simple two-armed version.
\subsection*{Two-armed Model:} We suppose the cams of each arm have only one mode and do not interfere with each other, that is, the non-Futurity payout modes for each arm are i.i.d. with parameter $p_A$ (resp., $p_B$). Let us take the $A$ arm as an example, we can set $I=J$ and $d=1$, and it is easy to verify that $\tilde{p}_i=p_A$ for all $i=0,1,\cdots,I-1$. If a gambler has been playing using one arm, then using the expression of $p^\circ$ of the one-armed model, the asymptotic probability $p^\circ_A$ (resp., $p^\circ_B$) that the gambler with $A$ (resp., $B$) arm wins the Futurity award has the following form from equation (\ref{eq1}):
\begin{align}\label{eq2}
p^\circ_A=\frac{p_Aq^J_A}{1-q^J_A},\qquad  \qquad \qquad  p^\circ_B=\frac{p_Bq^J_B}{1-q^J_B}.
\end{align}
Here, $p_A$ and $p_B$ respectively represent the probability of winning, using arm $A$ or arm $B$, at each coup  without the Futurity award and $q_A:=1-p_A$, $q_B:=1-p_B$.

\subsection*{Profitable ideas for the casino:} Parrondo game is that if a gambler randomly or periodically switches between two sets of fair games, the gambler can counterintuitively achieve a winning outcome under certain circumstances. Ethier and Lee \cite{ethier2010markovian} adopted Abbott's point \cite{abbott2009developments} of view to offer two ideas for how casinos can be profitable in the long run.

    \subsubsection*{Random-mixture strategy $C$:} One of the ideas is that the casino allows the gambler to randomly switch between the two arms. Specifically, the gambler chooses arm $A$ or $B$ by tossing a $\gamma$-coin before each coup; that is, the probability of heads on each toss is $\gamma$ where $0<\gamma<1$. The gambler plays the $A$ arm when the coin is heads and the $B$ arm when the coin is tails, regardless of previous results. Ethier and Lee called this the random-mixture strategy $C$. They showed that by this strategy, the casino can always make a profit in the long run. The following theorem is Theorem 4 in Section 6 of \cite{ethier2010markovian}.
\begin{theorem}[Ethier and Lee 2010]\label{EL}
Under random-mixture strategy $C$, the casino's asymptotic profit per coup $R_C$ attributable to the two-armed feature of the machine is for $J\ge 2$, $0<\gamma<1$,
 $$R_C=\gamma f(q_A)+(1-\gamma)f(q_B)-f(\gamma q_A+(1-\gamma)q_B),$$
 where $f(z)=J(1-z)z^J/(1-z^J)$, $q_A=1-p_A$, $q_B=1-p_B$.

 Furthermore, if $p_A\neq p_B$, then $R_C>0$ for any random-mixture strategy $C$. The Parrondo effect is observed for any random-mixture strategy $C$, which results from the fact that $f(\cdot)$ is a convex function.
\end{theorem}
\subsubsection*{Nonrandom periodic pattern strategy $D$:} An alternative concept entails the casino permitting the gambler to periodically switch between the two arms. The casino allows the gambler to pre-formulate a deterministic strategy that governs the frequency and timing of arm switching. Then, the gambler pulls the arm according to this strategy and repeats it indefinitely. The casino does not limit the number of times the gambler pulls the $A$ or $B$ arm or the order in which the arms are pulled; however, the strategy involves the pulling of $A$ and $B$ at least once. For example, in strategy $D=ABB$, the gambler will pull the $A$ arm once, then pull the $B$ arm twice, repeating that sequence indefinitely. Ethier and Lee defined this strategy as a (finite) nonrandom periodic pattern strategy $D$. We denote $r$ and $s$ as the numbers of $A$s and $B$s in strategy $D$, respectively. Ethier and Lee showed in \cite{ethier2010markovian} that if no restrictions are imposed on the slot machine, and gamblers are allowed to play according to any nonrandom periodic pattern strategy $D$, the casino cannot guarantee profitability in the long run. 

Herein, we mathematically compute the casino's asymptotic expected profit $R_D$ for different nonrandom periodic pattern strategies $D$ for $J=2$ and $p_A\neq p_B$ based on the condition in Ethier and Lee's conjectural assumption. Then we can make a new cam and let $I=2(r+s)$. For example, if $D=ABB$, then the new cam configuration is $ABBABB$. That is, a simple two-armed slot machine under strategy $D$ can be regarded as a generalized one-armed slot machine with a specific cam. The value of $p^\circ_D$ has the following form, which has been given by Equation (\ref{eq1}).
\begin{align}\label{eq3}
p^\circ_D=\frac{1}{r+s}\sum\limits_{k=1}^{r+s}\bigg(\sum\limits_{j=1}^{r+s}p_j\prod\limits_{i=j+1}^{j+2k-(r+s)\lfloor2k/(r+s)\rfloor}q_i\bigg)\frac{(q^r_Aq^s_B)^{\lfloor2k/(r+s)\rfloor}}{1-(q^r_Aq^s_B)^2}.
\end{align}
Here each of $p_i$ for $i=1,2,\cdots,r+s$ to be $p_A$ or $p_B$ in accordance with the corresponding term in the pattern $D$, and we can extend the definition of $p_i$ such that $p_{i+r+s}=p_i$ for any $i\in\{1,2,\cdots,r+s\}$, and $q_i:=1-p_i$ for $i=1,2,\cdots,2(r+s)$.
\begin{theorem}[Ethier and Lee 2010]
If $p_A\neq p_B$, $J\ge2$, and $r,s\ge1$, and if $r+s$ divides $J$, then the Parrondo effect is present for the nonrandom periodic pattern strategy $D$.
\end{theorem}
The conjecture of Ethier and Lee for the two-armed slot machine of this model is proved step by step in the following section. This conjecture was discussed on page 1098--1125 of Section 6 in \cite{ethier2010markovian}.

\begin{conjecture}[Ethier and Lee 2010] \label{cj1}
\textit{if the assumption that $r + s$ divides $J$ is replaced by the assumption that $J = 2$ (with $r \ge 1$ and $s \ge 1$), then the same conclusion holds.}
\end{conjecture}
\begin{remark}
We continue to extend the definition of $p_i$ such that $p_{i+3r+3s}=p_{i+2r+2s}=p_i$ for any $i\in\{1,2,\cdots,r+s\}$, and $q_i:=1-p_i$ for $i=1,2,\cdots,4r+4s$. After we extend the definition, $p_i$ can also be regarded as the probability of winning at the $i$th coup, ignoring the Futurity award. According to the proof of Theorem 5 in \cite{ethier2010markovian}, the presence of the Parrondo effect is equivalent to
\begin{align}\label{eq4}
p^\circ_D<\frac{rp^\circ_A+sp^\circ_B}{r+s}.
\end{align}
\end{remark}
Comparing the proof of Ethier and Lee's theorem in Theorem \ref{EL}, it is established that the asymptotic profit per coup $R_C$ is a convex function.
Our work is to use the above formula to more accurately calculate the asymptotic profit per coup $R_D$ of the casino and reveal the reason why the casino must be profitable in the long run from a mathematical point of view. We will give the asymptotic profit per coup $R_D$ of the casino under the nonrandom periodic pattern strategy $D$ in subsequent sections.

\section{Main results}\label{sec3}

Herein, we study how the nonrandom periodic pattern strategy $D$ affects the casino's asymptotic profit per coup. Following the strategy provided by players, the casino first analyzes the total number of executions of arms $A$ and $B$ in the player's single-round strategy, denoted as $r$ and $s$, respectively. Further, we examine the casino's asymptotic profits per coup under different strategies after fixing $r$ and $s$.

We then present our primary findings. While proving the conjecture of Ethier and Lee, we found additional characteristics of the Futurity slot machine under the nonrandom periodic pattern strategy $D$. Certain findings are not restricted by the condition $J=2$, which is helpful for studying the proof of other conjectures of Ethier and Lee.

As a preliminary finding, we observed that the casino's asymptotic profits per coup are equal for some strategies available to players. Thus, the following equivalent strategy can always express any strategy the player adopts.

\begin{definition}\label{de1}
Let $r$ and $s$ be any pair of positive fixed values. Then, any nonrandom periodic pattern strategy $D$ which switches between arms $2h$ times in a single cycle can be writed as
$$D=\underbrace{A\cdots A}_{r_1}\underbrace{B\cdots B}_{s_1}\underbrace{A\cdots A}_{r_2}\underbrace{B\cdots B}_{s_2}\cdots \underbrace{A\cdots A}_{r_h}\underbrace{B\cdots B}_{s_h}:=A^{r_1}B^{s_1}A^{r_2}B^{s_2}\cdots A^{r_h}B^{s_h},$$
where $r_k$ and $s_k$ are positive integers for $1\le k\le h$ and $\sum\limits_{k=1}^{h}r_k=r$, $\sum\limits_{k=1}^{h}s_k=s$. and the vector
$$\boldsymbol{a}(h):=(a_1,a_2,...,a_{2h})=(r_1,s_1,r_2,s_2,...,r_h,s_h).$$
For any fixed values $l>0$, we consider the vector $\boldsymbol{a}(l)=(a_1,a_2,...,a_{2l})$. We extend the definitions of the elements in the vector. In particular, for $1\le i \le 2l$, we define the $i$th component of $\boldsymbol{a}(l)$ as ${a_i}(l)$ and
$${a_i}(l)={a_{i-2l}}(l)={a_{i+2l}}(l).$$

We respectively define a function ${b_i(l)}$, $i\in \mathbb{Z}$, $-2l+1 \le i \le 4l$, as \\

  \qquad${b_{2j-1}}(l):=(-1)^{a_{2j-1}}q_A^{a_{2j-1}}$, ${b_{2j}}(l):=(-1)^{a_{2j}}q_B^{a_{2j}}$ for $-l+1\le j\le 2l$.\\

In particular, if $l=h$, we abbreviate $a_i(h)$ and $b_i(h)$ as $a_i$ and $b_i$, respectively.

\end{definition}
\begin{remark}
Equation (\ref{eq3}) can be rewritten as follows.
\begin{align}\label{eq5}
p^\circ_D=\frac{1}{r+s}\sum\limits_{k=1}^{r+s}\bigg(\sum\limits_{j=1}^{r+s}p_j\prod\limits_{i=j+1}^{j+2k}q_i\bigg)\frac{1}{1-(q^r_Aq^s_B)^2}.
\end{align}

 Any strategy $D$ used by the player can always be represented by the form defined above. We can start from the coup that selects $A$ and turn clockwise to end at the coup that selects $B$; we take this path as the main strategy of our research. In this way, we can use the $r_k$(resp. $s_k$) to represent the number of times that $A$(resp. $B$) has been executed in the strategy segment.
 
This result may be relaxed with no restriction of $J=2$. The proof idea remains the same. It may be of great help to study other conditions of Ethier and Lee's conjecture, such as $\min(r,s)=1$.

For the convenience of our later research, we extend the definitions of the $r_k$ and the $s_k$ to ensure that the cycle continues. In the follow-up proof, we construct this new function $b^D_i$ so as to be the key parameter that affects the casino's asymptotic profits per coup across various strategies.

\end{remark}
The first result is that we accurately calculate the asymptotic profits per coup of the casino for each nonrandom periodic pattern strategy $D$ developed by the player in advance, which enables the casino to adjust its profit according to the following theorem.
\begin{theorem}\label{th2}
For any set of fixed values $h,r,s>0$, with nonrandom periodic pattern strategy $D$ as in Definition \ref{de1}, the casino's asymptotic profit per coup $R_D$ attributable to the two-armed feature of the machine is
$$R_D=2QS,$$
 where
 $$Q=h+\sum\limits_{m=1}^{2h}\sum\limits_{j=1}^{2h-1}(-1)^j\prod\limits_{i=m}^{m+j-1}{b_i}+h\prod\limits_{i=1}^{2h}{b_i},$$
 $$S=\frac{(q_A-q_B)^2(1+(-1)^{r+s}q_A^rq_B^s)}{(r+s)(1+q_A)^2(1+q_B)^2(1-(q_A^rq_B^s)^2)}.$$
\end{theorem}

  \begin{remark}
The main proof of the theorem is in Section \ref{sec5} and the Appendix. For the simplest strategy $D$ with the least number of player arm changes, that is, $h=1$ and $D=A\cdots AB\cdots B$, we can directly calculate the casino's asymptotic profit per coup (Lemma \ref{le4}). Next, we can exchange part of the strategy segment $A$ and strategy segment $B$ to generate a new strategy and then calculate the difference between the asymptotic profit per coup of the casino obtained by adopting the new strategy and original strategy (Lemma \ref{le5}). Finally, any nonrandom periodic pattern strategy can be obtained from the simplest strategy by continuously exchanging strategy segment $A$ and strategy segment $B$. Thus, Lemma \ref{le5} and induction are used to obtain the form of the theorem. The specific proof is provided in the Appendix B.

  Notably, the casino's asymptotic profit per coup consists of the product of three parts: $2$, $Q$, and $S$. The first part, ``$2$", indicates that the casino must pay the player two coins for every two consecutive player losses. The second part, the function $Q$, represents the effect of arm playing order on the profit of the casino, with the player deciding to play $A$ $r$ times and then $B$ $s$ times. The third part, the function $S$, represents the influence of the internal parameters of the slot machine on the profit of the casino, given the values of $r$ and $s$ in the player's strategy.

After fixing the number $r$ of playing arm $A$ and the number $s$ of playing arm $B$, we conjecture that more frequent arm switches result in higher asymptotic profits per coup of the casino. We do not have a proof, and this seems to be an open problem.

  \end{remark}

The original conjecture of Ethier and Lee can be regarded as the following theorem.
\begin{theorem}\label{th3}
In the case where a player adopts a nonrandom periodic pattern strategy $D$ while playing
the two-armed version of the Futurity slot machine, the
casino's asymptotic profit per coup attributable to the two-armed feature of the
machine, $R_D$, is positive. This implies that the two-armed machine is more profitable than
two one-armed machines under the same strategy, and all nonrandom periodic pattern strategies $D$ are susceptible to the Parrondo effect.
\end{theorem}

  \begin{remark}
The asymmetry $p_A\neq p_B$ leads to $S>0$, and in Section \ref{sec6}, we show that $Q>0$, that is, the Parrondo effect is in effect for all nonrandom periodic pattern strategies $D$.

The proof of the theorem is more complicated, equivalent to an arrangement problem. The form of $Q$ is the sum of product terms, that is $(-1)^j\prod\limits_{i=m}^{m+j-1}{b^D_i}$. We try to rearrange $Q$ and write $Q$ into the form of several sums of positive product terms. Further, we use a typical example in Section \ref{sec6} to demonstrate the unique arrangement universally we found. The rigorous proof of the theorem is provided in Appendix C.
  \end{remark}

\section{Calculation of the asymptotic profit per coup for the casino} \label{sec5}

We note that $p^\circ_A$ and $p^\circ_B$ can be directly calculated by Equation\ (\ref{eq2}). If we want to prove Inequality (\ref{eq4}), we only need to figure out how strategy $D$ affects the value of $p^\circ_D$. Ethier and Lee stated in \cite{ethier2010markovian} that $p^\circ_D$ has the form (\ref{eq3}), but this seems too abstract. In fact, the gameplay in strategy $D$ will repeat indefinitely, so for any strategy $D$ we can always find a strategy $D'$ starting from arm $A$ such that $p^\circ_{D'}=p^\circ_D$. This way we can simplify our analysis. In the first part we will show why casinos can honestly claim that both arms are fair, and make a preliminary calculation of the casino's asymptotic profit per coup. In the second part, we find a simple example to calculate the value of $R_D$. Then, for any strategy, we find a similar strategy but with fewer arm swaps before repeating the pattern, and we try to calculate the difference in the casino's asymptotic profit per coup between these two strategies. Finally, we can indirectly calculate the casino's asymptotic profit per coup for any strategy.

\subsection{Preliminary calculation of casino asymptotic profits per coup}
We build on Pyke's assumptions \cite{pyke2003random}, without loss of generality, assume that the casino will claim that both arms are fair. Before calculating the asymptotic expected profit of the casino, we must determine the parameter structure of the fair arm in the game. If a single arm is fair, for example, the $A$ arm, this means that the asymptotic payoff per coup $\mu_A^*$ of a player playing the $A$ arm each time is exactly 1 coin consumed per coup, specifically
$$\mu_A^*=p_Au_A+2p^\circ_A=1,$$
where $u_A$ is the conditional expected (non-Futurity) return from arm $A$, given a nonzero return. We use Equation (\ref{eq2}) mentioned by Ethier and Lee to obtain a one-dimensional quadratic equation with $p_A$ as the dependent variable. Solving this equation yields the parameter $u_A=\frac{3-2p_A}{2-p_A}$. In other words, we can adjust parameter $p_A$ and accordingly adjust the value of $u_A$ to always guarantee the fairness of this arm. The casino can also set the $B$ arm to be fair, but it needs to make $p_A\neq p_B$. Then the asymptotic profit of the casino is actually the 1 coin that the player spends before each coup minus the asymptotic payoff per coup $\mu^*_D$ that the player obtains when playing the game according to strategy $D$. With these parameters, we can make an initial calculation of the asymptotic profit per coup of the casino.
\begin{lemma}\label{le3}
For positive fixed integers $r$ and $s$ under nonrandom periodic pattern strategy $D$, the casino's asymptotic profit per coup $R_D$ attributable to the two-armed feature of the machine is
$$R_D=R_0-2Q_0S_0,$$
where
\begin{align}\label{eq6}
R_0=\frac{2}{r+s}\bigg(\frac{rq^2_A}{1+q_A}+\frac{sq^2_B}{1+q_B}+(-1)^{r+s}(r+s)(rq^{r+1}_Aq^s_B+sq^r_Aq^{s+1}_B)S_0\bigg),
\end{align}
\begin{align}\label{eq7}
Q_0=\sum_{j=1}^{r+s}\sum_{k=2}^{r+s}(-1)^k\prod\limits_{i=j}^{j+k-1}q_i,\qquad\qquad S_0=\frac{1+(-1)^{r+s}q^r_Aq^s_B}{(r+s)(1-(q^r_Aq^s_B)^2)}.
\end{align}
\end{lemma}
\begin{proof} For the law of large numbers, we know that
$$\mu_D=\frac{r}{r+s}p_Au_A+\frac{s}{r+s}p_Bu_B.$$
Then according to the above discussion to obtain
\begin{align}\label{eq9}
R_D=\frac{r}{r+s}\mu^*_A+\frac{s}{r+s}\mu^*_B-\mu_D-2p^\circ_D=2\bigg(\frac{r}{r+s}p^\circ_A+\frac{s}{r+s}p^\circ_B-p^\circ_D\bigg).
\end{align}
 where $\mu_D$ is the asymptotic mean payout per coup, disregarding the Futurity award. Recall that we computed $p^\circ_D$, which has the form (\ref{eq5}), and our main concern was its internal summation, that is,
\begin{align*}
\sum\limits_{k=1}^{r+s}\bigg(\sum\limits_{j=1}^{r+s}p_j\prod\limits_{i=j+1}^{j+2k}q_i\bigg)&=\sum\limits_{k=1}^{r+s}\sum\limits_{j=1}^{r+s}\prod\limits_{i=j+1}^{j+2k}q_i-\sum\limits_{k=1}^{r+s}\sum\limits_{j=1}^{r+s}\prod\limits_{i=j}^{j+2k}q_i\\
&=\sum\limits_{k=1}^{r+s}\sum\limits_{j=2}^{r+s}\prod\limits_{i=j}^{j+2k-1}q_i+\sum\limits_{k=1}^{r+s}\prod\limits_{i=r+s+1}^{r+s+2k}q_i-\sum\limits_{k=1}^{r+s}\sum\limits_{j=1}^{r+s}\prod\limits_{i=j}^{j+2k}q_i\\
&=\sum\limits_{k=1}^{r+s}\sum\limits_{j=1}^{r+s}\prod\limits_{i=j}^{j+2k-1}q_i-\sum\limits_{k=1}^{r+s}\sum\limits_{j=1}^{r+s}\prod\limits_{i=j}^{j+2k}q_i\\
&=\sum\limits_{k=2}^{2(r+s)+1}\sum\limits_{j=1}^{r+s}(-1)^k\prod\limits_{i=j}^{j+k-1}q_i.
\end{align*}
We observe that some $k$ values in the parameter $k$ in the summation formula have similar product structures, and for some special $k$ values, we can directly calculate the result of the summation term regardless of the structure of strategy $D$.
If $2\le k\le r+s+1$, we consider $k'=k+r+s$. Then we have
$$\sum\limits_{j=1}^{r+s}(-1)^{k'}\prod\limits_{i=j}^{j+k'-1}q_i=(-1)^{r+s}q^r_Aq^s_B\bigg(\sum\limits_{j=1}^{r+s}(-1)^{k}\prod\limits_{i=j}^{j+k-1}q_i\bigg).$$
In particular, if $k=r+s+1$, then we have
\begin{align*}
\sum\limits_{j=1}^{r+s}(-1)^{r+s+1}\prod\limits_{i=j}^{j+r+s}q_i&=\bigg(\sum\limits_{j=1}^{r+s}q_j\bigg)(-1)^{r+s+1}\prod\limits_{i=j+1}^{j+r+s}q_i\\
&=(-1)^{r+s+1}(rq^{r+1}_Aq^s_B+sq^r_Aq^{s+1}_B).
\end{align*}
Then, the equation below (\ref{eq9}) can be simplified to
\begin{align*}
\sum\limits_{k=1}^{r+s}\bigg(\sum\limits_{j=1}^{r+s}p_j\prod\limits_{i=j+1}^{j+2k}q_i\bigg)&=(1+(-1)^{r+s}q^r_Aq^s_B)\bigg(\sum\limits_{k=2}^{r+s}\sum\limits_{j=1}^{r+s}(-1)^{k}\prod\limits_{i=j}^{j+k-1}q_i\bigg)\\
&\qquad{}+((-1)^{r+s+1}-q^r_Aq^s_B)(rq^{r+1}_Aq^s_B+sq^r_Aq^{s+1}_B).\\
\end{align*}
Now we substitute the above equation and equation (\ref{eq2}) into equation (\ref{eq9}) to obtain
$$R_D=R_0-2Q_0S_0,$$
where $R_0$ is equation (\ref{eq6}) and $Q_0$, $S_0$ are equations (\ref{eq7}). 
\end{proof}

\begin{remark}
In the process of calculating $R_D$, we only pay attention to the changes in the value of $Q_0$. We need to consider the probability product of all adjacent coup failures from the first step of a single strategy to the last step of the strategy, where the number of product terms is from $2$ to $r+s$, and if the number of product terms is even, it is positive, and if the number of product terms is odd, it is negative. Although we have reduced the number of product terms from $2(r+s)+1$ to $r+s$, we still cannot directly find the relationship between strategy $D$ and the value of $R_D$. In the following section, we use the lemma to directly calculate the asymptotic profit per coup of the strategy $D=A\cdots AB\cdots B$ that involves swapping arms once before repeating the pattern.
\end{remark}

\subsection{Asymptotic profit per coup for the casino under a simple strategy}A very natural strategy that players can think of is $D=AA\cdots ABB\cdots B$, which happens to be calculated directly using Lemma \ref{le3}. Let's first look at the simplest case, which is $D=AB$; that is $r,s=1$, then we have
$$Q_0=\sum_{j=1}^{2}\prod\limits_{i=j}^{j+1}q_i=2q_Aq_B.$$
Further, if the player applies strategy $D=A^rB$, that is, $r>1$, $s=1$, and by equation ($\ref{eq7}$), then we have
\begin{align*}
Q_0&=\sum_{k=2}^{r+1}(-1)^k\bigg(\sum_{j=1}^{r-k+1}\prod\limits_{i=j}^{j+k-1}q_i+\sum_{j=r-k+2}^{r+1}\prod\limits_{i=j}^{j+k-1}q_i\bigg)\\
&=\sum_{k=2}^{r+1}(-1)^k((r-k+1)q^k_A+kq^{k-1}_Aq_B).
\end{align*}
Similarly, the case $r=1$, $s>1$ leads to
$$Q_0=\sum_{k=2}^{1+s}(-1)^k(kq_Aq^{k-1}_B+(s-k+1)q^k_B).$$
Finally, if the player applies strategy $D=A^rB^s$, that is, $r,s>1$. We observe that $k$ on the left side of equation (\ref{eq7}) represents the number of factors in the product term, which we call the ``length" of the product term. We start from the first step of the strategy and observe the number of factor $q_A$ and factor $q_B$ in the product term of ``length" $k$, and then we continue with the second step according to the rules just mentioned until the $r+s$ step. We record the number of $i$ $q_A$ factors and $j$ $q_B$ factors contained in the product term of ``length" $i+j$, and denote it as $M(i,j)$. Then, we naturally get the following equation.
\begin{align*}
Q_0:=\sum_{i=0}^{r}\sum_{j=0}^{s}(-1)^{i+j}M(i,j)q^i_Aq^j_B.
\end{align*}
 For our simple example, we only need to consider the following case.
\begin{enumerate}
\item If $i=0$, $j>1$, then $M(0,j)=s-j+1$, and similarly if $j=0$, $i>1$, we have $M(i,0)=r-i+1$.
\begin{align*}
A\cdots AB\cdots B\underbrace{B\cdots B}_jB\cdots B,\qquad\qquad\qquad A\cdots A\underbrace{A\cdots A}_iA\cdots AB\cdots B.
\end{align*}
\item If $0<i<r$ and $0<j<s$, then $M(i,j)=2$, and the two cases are as follows.
$$A\cdots A\rlap{$\overbrace{\phantom{AAAAAA}}^k$}\underbrace{A\cdots A}_i  \underbrace{B\cdots B}_jB\cdots B\rlap{$\overbrace{\phantom{AAAAAA}}^k$}\underbrace{B\cdots B}_j\underbrace{A\cdots A}_iA\cdots A.$$
\item If $i=r$, $j<s$, then $M(r,j)=j+1$, and similarly if $j=s$, $i<r$, we have $M(i,s)=i+1$.
\begin{align*}
&A\cdots AB\cdots B\overbrace{B\cdots B\underbrace{A\cdots A}_rB\cdots B}^{r+j}B\cdots B,\\
&A\cdots A\overbrace{A\cdots A\underbrace{B\cdots B}_sA\cdots A}^{i+s}A\cdots A.
\end{align*}
\end{enumerate}
For a more intuitive representation, we use a matrix $\boldsymbol{M}$ to represent the count parameter $M(i,j)$.
$$
\boldsymbol{M}=
\begin{bmatrix}
  0 & 0 & s-1 & \cdots & 2 & 1\\
  0 & 2 & 2 & \cdots & 2 & 2 \\
  r-1 & 2  & 2 & \cdots & 2& 3  \\
  \vdots& \vdots  & \vdots  &  \ddots & \vdots & \vdots \\
  2 & 2 & 2 & \cdots & 2 & r \\
  1 & 2 & 3 & \cdots & s & r+s
\end{bmatrix}
=(M(i,j))_{(r+1)\times(s+1)}.
$$
Based on this typology, we can calculate the casino's asymptotic profit per coup under a player strategy of swapping arms only twice before repeating the pattern. We then have the following lemma.

\begin{lemma}\label{le4}
For any pair of fixed values $r,s>0$, under a nonrandom periodic pattern player strategy $D$, where $D=A^rB^s$, then the casino's asymptotic profit per coup $R_D$ attributable to the two-armed feature of the machine is
$$R_D=2Q_1S,$$
where 
\begin{align}\label{Q_1}
Q_1=(1-(-1)^rq_A^r)(1-(-1)^sq_B^s)
\end{align}
 and
 \begin{align}\label{S}
  S=\frac{(q_A-q_B)^2(1+(-1)^{r+s}q_A^rq_B^s)}{(r+s)(1+q_A)^2(1+q_B)^2(1-(q_A^rq_B^s)^2)}.
  \end{align}
\end{lemma}

\begin{proof} Based on the above discussion, if $r=1$ or $s=1$, then by Lemma \ref{le3} our conclusion is trivial.
For fixed $r,s>1$, we note that the upper left element of the count matrix is 0. For the convenience of subsequent calculations, we add some elements to make certain elements of this matrix correspond to each other, namely,
\begin{align}\label{eq10}
 Q_0&=\sum_{i=0}^{r}\sum_{j=0}^{s}(-1)^{i+j}M(i,j)q^i_Aq^j_B \nonumber\\
&:=\sum_{i=0}^{r}\sum_{j=0}^{s}M_1(i,j)q^i_Aq^j_B-(r+s)+rq_A+sq_B,
\end{align}
where $\boldsymbol{M_1}=(M_1(i,j))_{(r+1)\times(s+1)}$, i.e.,
$$
\boldsymbol{M_1}=
\begin{bmatrix}
  r+s & -s & s-1 & \cdots & (-1)^{s-1}\times2 & (-1)^{s}\\
  -r & 2 & -2 & \cdots & (-1)^{s}\times2 & (-1)^{s+1}\times2 \\
  r-1 & -2  & 2 & \cdots & (-1)^{s+1}\times2& (-1)^{s+2}\times3  \\
  \vdots& \vdots  & \vdots  &   & \vdots & \vdots \\
  (-1)^{r-1}\times2 & (-1)^{r}\times2 & (-1)^{r+1}\times2 & \cdots & (-1)^{r+s-2}\times2 & (-1)^{r+s-1}r \\
  (-1)^{r} & (-1)^{r+1}\times2 & (-1)^{r+2}\times3 & \cdots & (-1)^{r+s-1} s & (-1)^{r+s}(r+s)
\end{bmatrix}.
$$
As we continue to process the first summation term in equation (\ref{eq10}), we take advantage of the symmetry of the coefficient matrix, obtaining
$$\sum_{i=0}^{r+1}\sum_{j=0}^{s+1}M_2(i,j)q^i_Aq^j_B=\sum_{i=0}^{r}\sum_{j=0}^{s}M_1(i,j)q^i_Aq^j_B(1+q_A)(1+q_B),$$
where $\boldsymbol{M_2}=(M_2(i,j))_{(r+2)\times(s+2)}$, i.e.,
$$
\boldsymbol{M_2}=
\begin{bmatrix}
  r+s & r & -1 & 1 & \cdots & (-1)^{s-2}& (-1)^{s-1} & (-1)^s \\
  s & 2 & -1 & 1 & \cdots & (-1)^{s-2}& (-1)^{s-1} & (-1)^{s+1} \\
  -1 & -1 & 0 & 0 & \cdots &0& (-1)^{s+2} & (-1)^{s+2} \\
  1 & 1 & 0 & 0 & \cdots &0& (-1)^{s+3} & (-1)^{s+3}\\
  \vdots & \vdots  & \vdots & \vdots & \ddots  &\vdots& \vdots & \vdots \\
  (-1)^{r-2}& (-1)^{r-2}  & 0  & 0  & \cdots &0& (-1)^{r+s-1} & (-1)^{r+s-1} \\
  (-1)^{r-1} & (-1)^{r+1} & (-1)^{r+2} & (-1)^{r+3} & \cdots & (-1)^{r+s-1}& (-1)^{r+s}\times2 & (-1)^{r+s}s\\
  (-1)^r & (-1)^{r+1} & (-1)^{r+2} & (-1)^{r+3} & \cdots & (-1)^{r+s-1}& (-1)^{r+s}r & (-1)^{r+s}(r+s)
\end{bmatrix}.
$$
In the new coefficient matrix $\boldsymbol{M_2}$, if we extract the three elements in the upper left corner and the three elements in the lower right corner, then we can build yet another coefficient matrix $\boldsymbol{M_3}$ containing the common factor $(q_A-q_B)^2$. We note this factor in the formula
\begin{align}\label{eq11}
 &\sum_{i=0}^{r+1}\sum_{j=0}^{s+1}M_2(i,j)q^i_Aq^j_B=\sum_{i=0}^{r-1}\sum_{j=0}^{s-1}M_3(i,j)q^i_Aq^j_B(q_A-q_B)^2\nonumber\\
 &\qquad{}+r+s+sq_A+rq_B+(-1)^{r+s}(rq^{r+1}_Aq^{s}_B+sq^r_Aq^{s+1}_B+(r+s)q^{r+1}_Aq^{s+1}_B)\nonumber\\
&=-(1+q_A)^{-1}(1+q_B)^{-1}(1-(-1)^rq_A^r)(1-(-1)^sq_B^s)(q_A-q_B)^2+r+s\nonumber\\
 &\qquad{}+sq_A+rq_B+(-1)^{r+s}(rq^{r+1}_Aq^{s}_B+sq^r_Aq^{s+1}_B+(r+s)q^{r+1}_Aq^{s+1}_B),
\end{align}
where $\boldsymbol{M_3}=(M_3(i,j))_{r\times s}$, i.e.,
 $${\boldsymbol{M_3}}:=
  \begin{bmatrix}
  -1 & 1 & -1 & \cdots & (-1)^{s-1} & (-1)^{s}\\
  1 & -1 & 1 & \cdots & (-1)^{s} & (-1) ^{s+1}\\
  \vdots & \vdots & \vdots &  \  & \vdots & \vdots\\
  (-1)^{r-1} & (-1)^{r}  & (-1)^{r+1} & \cdots &(-1)^{r+s-3} & (-1)^{r+s-2}\\
  (-1)^{r} & (-1)^{r+1} &(-1)^{r+2} & \cdots & (-1)^{r+s-2} & (-1)^{r+s-1}
\end{bmatrix}.
$$
Then by Lemma \ref{le3} and Equations (\ref{eq6}), (\ref{eq7}), (\ref{eq10}), and (\ref{eq11}), it is easy to verify that
$$R_D=R_0-2Q_0S_0=2Q_1S,$$
where $Q_1$ is equation (\ref{Q_1}) and $S$ is equation (\ref{S}). 
\end{proof}

\begin{remark}
For any fixed pair of values $r,s>0$, if $D=A^rB^s$, then $h=1$ based on Definition \ref{de1}. The reason why we can easily find the values of the elements in $\boldsymbol{M}$ is that when $h=1$, $D$ has a unique strategy. That is to say, the elements of $\boldsymbol{M}$ can be uniquely determined. When $h=2$, there are two strategies. Although we cannot uniquely determine the value of the elements in $\boldsymbol{M}$ for $h=2$, the two strategies are equivalent. But when $h>2$, the multiple strategies are not necessarily equivalent. Therefore, when $h$ is large, it is difficult for us to directly calculate the casino's asymptotic profit per coup $R_D$.
\end{remark}

\subsection{Differences between the two strategies}  It is difficult to directly calculate the values of the elements of $\boldsymbol{M}$, but we can find the difference in $R_D$ between two specified strategies. Before this, we first study the characteristics of the difference between the parameter matrices of any two strategies. Recalling Definition \ref{de1}, for any fixed $r,s>0$, any strategy $D$ has an equivalent strategy in the form $$D=A^{r_1}B^{s_1}\cdots A^{r_{h-1}}B^{s_{h-1}}A^{r_h}B^{s_h}.$$
We treat each set of consecutive pulls of the same arm as a strategy segment such as $A^{r_1}$. Treated this way, this strategy has a total of $2h$ strategy segments. We then try to exchange any two adjacent strategies in $D$ to form a new strategy $D'$.  To facilitate our description later, we exchange the $(2h-2)$-th and the $(2h-1)$-th segments in the strategy, that is,
$$D'=A^{r_1}B^{s_1}\cdots A^{r_{h-1}}A^{r_h}B^{s_{h-1}}B^{s_h}.$$

\begin{lemma}\label{le5}
For any fixed $h,r,s>0$, we consider two nonrandom-pattern strategies $D$, $D'$.
  $$D=A^{r_1}B^{s_1}\cdots A^{r_{h-1}}B^{s_{h-1}}A^{r_h}B^{s_h},$$  $$D'=A^{r_1}B^{s_1}\cdots A^{r_{h-1}}A^{r_h}B^{s_{h-1}}B^{s_h}.$$
Here $1\le k\le h$, $r_k>0$, $s_k>0$, $\sum\limits_{k=1}^{h}r_k=r$, and $\sum\limits_{k=1}^{h}s_k=s$.

Then, the difference between the casino's asymptotic profits per coup for the two strategies is
$$R_D-R_{D'}=2S(1-{b_{2h-2}})(1-{b_{2h-1}})\bigg(\sum\limits_{j=0}^{2h-3}(-1)^{j}\bigg(\prod\limits_{i=0}^{j-1}{b_i}+\prod\limits_{i=j}^{2h-3}{b_i}\bigg)\bigg),$$
where $$S=\frac{(q_A-q_B)^2(1+(-1)^{r+s}q_A^rq_B^s)}{(r+s)(1+q_A)^2(1+q_B)^2(1-(q_A^rq_B^s)^2)}.$$
\end{lemma}

\begin{remark}
Now, for any strategy $D(r,s)$, we can always exchange in the way of the above lemma to produce a new strategy of reduced $h$ value. In this way, any strategy can eventually become the strategy of $h=1$ mentioned in Lemma \ref{le4}, allowing us to indirectly calculate the value of $R_D$ of any strategy $D$.
\end{remark}

\section{Proof of the conjecture} \label{sec6}
Theorem \ref{th3} provides the precise value of the casino's asymptotic profit per coup. To assure the two-armed machine is more profitable than two one-armed machines under the same strategy, we must make $R_D>0$, and the sign of $R_D$ is only related to the parameter $Q$.
$$Q=\sum\limits_{l=1}^{h}\sum\limits_{j=0}^{2h-1}(-1)^{j}\prod\limits_{k=2l+1}^{2l+j}{b_k}+\sum\limits_{l=1}^{h}\sum\limits_{j=0}^{2h-1}(-1)^{j+1}\prod\limits_{k=2l}^{2l+j}{b_k}.$$

This section gives a typical example to reflect the sign change of parameter $Q$, and the strict and detailed proof is in Appendix C. From definition \ref{de1}, we recall the vector
$$\boldsymbol{a}=(a_1,a_2,\cdots,a_{2h}):=(r_1,s_1,r_2,s_2,\cdots,r_h,s_h)$$ 
to represent the strategy $D$.
We know that when $a_k$ is an even number, $0<b_k<1$. When $a_k$ is an odd number, $-1<b_k<0$. That is, the sign of $Q$ is only related to that of $b_k$, and that of $b_k$ is only related to the parity of $a_k$. It is sufficient to examine the effect of the parameter $a_k$ on the sign of $Q$. Let us prove that $Q>0$ using the strategy parameter $\boldsymbol{a}=(1,1,1,2,1,3)$ that is adopted by the player.
\begin{lemma}
In case where a player uses a nonrandom periodic pattern strategy $D=ABABBABBB$ while playing a two-armed slot machine, then the
casino's asymptotic profit per coup attributable to the two-armed feature of the
machine, $R_D$, is positive.
\end{lemma}
\begin{proof} In this case,
\begin{align}
 R&_{(1,1,1,2,1,3)}\label{line1}\\
=&R_{(1,1,1,2,1,3)}-R_{(1,1,2,5)}\label{line2}\\
&+R_{(1,1,2,5)}-R_{(3,6)}\label{line3}\\
&+R_{(3,6)}\label{line4}.
\end{align}
Lines (\ref{line2}) and (\ref{line3}) can be found from Lemma \ref{le5}, and line (\ref{line4}) can be found from
Lemma \ref{le4}. Alternatively, line (\ref{line1}) can be found from Theorem \ref{th2}.

In more detail, with $x = q_A$ and $y = q_B$,
$$R_{(3,6)}=2S(1+x^3)(1-y^6),$$
\begin{align*}
R_{(1,1,2,5)}-R_{(3,6)}=2S&(1-b_2)(1-b_3)(1+b_0b_1-(b_0+b_1))\\
=2S&(1+y)(1-x^2)(1+x)(1+y^5),
\end{align*}
where $b_1=-x$, $b_2=-y$, $b_3=x^2$, and $b_4=-y^5$,  extended periodically with
period $4$, and
\begin{align*}
R_{(1,1,1,2,1,3)}-R_{(1,1,2,5)}=2S&(1-b_4)(1-b_5)((1+b_0b_1b_2b_3)-(b_0+b_1b_2b_3)\\
&+(b_0b_1+b_2b_3)-(b_0b_1b_2+b_3))\\
=2S&(1+x)^2(1+xy)(1-y^2)(1+y^3)
\end{align*}
where $b_1=-x$, $b_2=-y$, $b_3=-x$, $b_4=y^2$, $b_5=-x$, and $b_6=-y^3$, extended periodically with period $6$.

We conclude that
\begin{align*}
R_{(1,1,1,2,1,3)}=2S&[(1+x)^2(1+xy)(1-y^2)(1+y^3)\\
&+(1+y)(1-x^2)(1+x)(1+y^5)+(1+x^3)(1-y^6)]>0,
\end{align*}
which can also be derived directly from Theorem \ref{th2}:
\begin{align*}
R_{(1,1,1,2,1,3)}=2S&[(1-b_1+b_1b_2-b_1b_2b_3+b_1b_2b_3b_4-b_1b_2b_3b_4b_5)\\
&+(-b_2+b_2b_3-b_2b_3b_4+b_2b_3b_4b_5-b_2b_3b_4b_5b_6+b_2b_3b_4b_5b_6b_1)\\
&+(1-b_3+b_3b_4-b_3b_4b_5+b_3b_4b_5b_6-b_3b_4b_5b_6b_1)\\
&+(-b_4+b_4b_5-b_4b_5b_6+b_4b_5b_6b_1-b_4b_5b_6b_1b_2+b_4b_5b_6b_1b_2b_3)\\
&+(1-b_5+b_5b_6-b_5b_6b_1+b_5b_6b_1b_2-b_5b_6b_1b_2b_3)\\
&+(-b_6+b_6b_1-b_6b_1b_2+b_6b_1b_2b_3-b_6b_1b_2b_3b_4+b_6b_1b_2b_3b_4b_5)]\\
=2S&[(1-b_1)(1-b_2b_3b_4b_5b_6)\\
&+(1-b_2+b_1b_2-b_6+b_6b_1-b_6b_1b_2)(1-b_3b_4b_5)\\
&+(-b_3+b_2b_3-b_1b_2b_3+b_6b_1b_2b_3)(1-b_4)\\
&+(1-b_5+b_5b_6-b_5b_6b_1+b_5b_6b_1b_2-b_5b_6b_1b_2b_3)(1-b_4)]>0
\end{align*}

According to Theorem \ref{th2}, if $p_A\neq p_B$, $Q>0$, then $R_D>0$, implying that the Parrondo effect is observed for the nonrandom periodic pattern strategy $D$.
\end{proof}
\begin{remark}
In our proof, we used two methods. The first one is based on Lemma \ref{le5}, looking for similar strategies with fewer arm changes, and then calculating the difference in the casino's asymptotic profit per coup between them, which is generally positive. This calculation will split the sum into several positive terms and add them together, thus proving the conjecture. But this method will fail in some special cases.
The other method is based on Theorem \ref{th2}, directly calculating the casino's asymptotic profit per coup, and adding the negative terms in the sum to the unique corresponding positive terms to form a positive polynomial, forming an iterative relationship. Or the positive value of the sum is added to the unique corresponding negative term to form a negative polynomial. This type of negative polynomial can always find a positive polynomial with the same iterative structure, so that their addition is a new positive polynomial. This way of finding iterative relationships is the key to our proof of the theorem, and detailed description is in Appendix C.
\end{remark}
It has been determined that a two-armed slot machine has the potential to generate long-term profits for a casino. From Lemma \ref{le3}, we know that $R_D>0$ is equivalent to obtaining Inequality (\ref{eq4}); hence, it is evident that Parrondo's paradox will occur at this time. Specifically, a game on the futurity of a two-armed slot machine in which one arm is engaged can be considered a fair game. The player plays repeatedly in some fixed order and will lose in the long run. In particular, no matter how the player formulates a strategy, it will become an unfair game, and the Parrondo effect will occur.

The Futurity award setting is the most notable feature of the Futurity slot machine. This game that records the number of consecutive failures of gamblers is strikingly similar to the original history-dependent Parrondo games \cite{parrondo2000new}. In our example, the limiting condition $J=2$ specifies that the player's payoff distribution before each coup is related to the player's winning or losing situation in the previous coup.

In Section \ref{sec5}, we clearly derive the casino's asymptotic profit per coup; thus, the casino can improve the two-armed slot machine by adjusting the winning probability of the two arms appropriately. Following our extremely complicated proof in Section \ref{sec5}, we have mathematically deciphered the secret of a casino's long-term profitability. We hope that gamblers are cautious when facing the futurity of a two-armed slot machine that casinos claim to be fair.

\section*{Appendix A}

\begin{proof} [Proof of Lemma \ref{le5}]
For the convenience of subsequent writing, we define the following functions.
\begin{equation}\label{eq12}
\begin{aligned}
\xi^{D'}(a,b;c,d):&=\sum_{x=a+1}^{b}\sum_{y=c+1}^{d}(-1)^{y-x+1}\prod\limits_{i=x}^{y}q^{D'}_i\\
\xi^{D}(a,b;c,d):&=\sum_{x=a+1}^{b}\sum_{y=c+1}^{d}(-1)^{y-x+1}\prod\limits_{i=x}^{y}q^{D}_i\\
\Delta\xi(a,b;c,d):&=\xi^{D'}(a,b;c,d)-\xi^{D}(a,b;c,d).
\end{aligned}
\end{equation}
Based on Lemma \ref{le3}, we observe the structures of $D$ and $D'$ and study each one's product term from equation (\ref{eq7}), that is,
\begin{align}\label{eq13}
 Q^{D'}_0-Q^D_0&=\sum_{j=1}^{r+s}\sum_{k=2}^{r+s}(-1)^k\bigg(\prod\limits_{i=j}^{j+k-1}q^{D'}_i-\prod\limits_{i=j}^{j+k-1}q^{D}_i\bigg)\nonumber\\
:&=\sum_{x=1}^{r+s}\sum_{y=x+1}^{r+s+x-1}(-1)^{y-x+1}\prod\limits_{i=x}^{y}(q^{D'}_i-q^{D}_i)\nonumber\\
&=\Delta\xi(0,r+s;x,r+s+x-1).
\end{align}
Then we need to study only the values of the product term $(-1)^{y-x+1}\prod\limits_{i=x}^{y}q_i$ under different strategies. If the elements involved in a given product term do not contain the part in which the strategy segments are swapped, then clearly the values of the product term are equal under the two strategies, as follows.
$$D\to A\cdots \overbrace{\cdots B\cdots A\cdots }^{product\ term}\cdots \overbrace{\underbrace{B\cdots B}_{s_{h-1}}\underbrace{A\cdots A}_{r_h}}^{swap \ \ segments}\underbrace{B..B\rlap{$\overbrace{\phantom{AAAAAA}}^{product \ \ term}$}B..B}_{s_{h}}\cdots A\cdots \ \cdots \overbrace{\underbrace{B\cdots B}_{s_{h-1}}\underbrace{A\cdots A}_{r_h}}^{swap \ \ segments}\underbrace{B\cdots B}_{s_{h}},$$
$$D\to A\cdots \overbrace{\cdots B\cdots A\cdots }^{product\ term}\cdots \overbrace{\underbrace{A\cdots A}_{r_h}\underbrace{B\cdots B}_{s_{h-1}}}^{swap \ \ segments}\underbrace{B..B\rlap{$\overbrace{\phantom{AAAAAA}}^{product \ \ term}$}B..B}_{s_{h}}\cdots A\cdots \ \cdots \overbrace{\underbrace{A\cdots A}_{r_h}\underbrace{B\cdots B}_{s_{h-1}}}^{swap \ \ segments}\underbrace{B\cdots B}_{s_{h}}.$$
Then if $0< x <\sum\limits_{k=1}^{2h-3}a_k$, $x<y\le \sum\limits_{k=1}^{2h-3}a_k$, or $\sum\limits_{k=1}^{2h-1}a_k< x <\sum\limits_{k=1}^{2h}a_k=r+s$, $x<y\le \sum\limits_{k=1}^{4h-3}a_k$, then we have $\prod\limits_{i=x}^{y}q^{D'}_i=\prod\limits_{i=x}^{y}q^{D}_i$. On the other hand, if our product term includes all the elements in the swap segments, then the product of the same two is equal, as follows.
$$D\to A\cdots A\cdots \overbrace{\cdots \underbrace{B\cdots BA\cdots A}_{swap \ \ segments}\cdots }^{product\ \ term}\cdots B\cdots B,$$
$$D'\to A\cdots A\cdots \overbrace{\cdots \underbrace{A\cdots AB\cdots B}_{swap \ \ segments}\cdots }^{product\ \ term}\cdots B\cdots B.$$
Therefore, if $0< x \le \sum\limits_{k=1}^{2h-3}a_k$, $\sum\limits_{k=1}^{2h-1}a_k<y<r+s+x$, or if $\sum\limits_{k=1}^{2h-1}a_k< x \le r+s$, $\sum\limits_{k=1}^{4h-1}a_k\le y<r+s+x$, then we have $\prod\limits_{i=x}^{y}q^{D'}_i=\prod\limits_{i=x}^{y}q^{D}_i$.\\
Then by Equations (\ref{eq12}), (\ref{eq13}) and for the difference between two strategies $D'$, $D$ in intervals $0< x \le \sum\limits_{k=1}^{2h-3}a_k$ and $\sum\limits_{k=1}^{2h-1}a_k<x\le r+s$, we have
\begin{align}\label{eq14}
&\Delta\xi\bigg(0,\sum\limits_{k=1}^{2h-3}a_k;x,r+s+x-1\bigg)+\Delta\xi\bigg(\sum\limits_{k=1}^{2h-1}a_k,r+s;x,r+s+x-1\bigg)\nonumber\\
&=\Delta\xi\bigg(\sum\limits_{k=1}^{2h-1}a_k,\sum\limits_{k=1}^{4h-3}a_k;\sum\limits_{k=1}^{4h-3}a_k,\sum\limits_{k=1}^{4h-1}a_k\bigg)\nonumber\\
&=\xi^D\bigg(\sum\limits_{k=1}^{2h-1}a_k,\sum\limits_{k=1}^{4h-3}a_k;\sum\limits_{k=1}^{4h-3}a_k-1,\sum\limits_{k=1}^{4h-3}a_k\bigg)\Delta\xi\bigg(\sum\limits_{k=1}^{4h-3}a_k,\sum\limits_{k=1}^{4h-3}a_k+1;\sum\limits_{k=1}^{4h-3}a_k,\sum\limits_{k=1}^{4h-1}a_k\bigg).
\end{align}
From the above formula, we observe that the sum of all product terms can be written as the product of the sum of two product terms. We note that each strategy segment is a proportional sequence sum if the first product term retrieves the strategy in the direction of the arrow as follows.
$$D\to A\cdots A\cdots \overset{\sum\limits_{k=1}^{2h-1}a_k}{\mid}\overset{\longleftarrow}{\underbrace{B\cdots B}_{s_h}}\overset{\longleftarrow}{\underbrace{A\cdots A}_{r_1}}\cdots \overset{\longleftarrow}{\underbrace{B\cdots B}_{s_{h-2}}}\overset{\longleftarrow}{\underbrace{A\cdots A}_{r_{h-1}}}\overset{\sum\limits_{k=1}^{4h-3}a_k}{\mid}\cdots $$
The value marked directly above the vertical line signifies that the strategy has reached the corresponding coup at this time. This is a more clear method used to represent the sum of the geometric sequence. Then,
\begin{align}\label{eq15}
&\xi^D\bigg(\sum\limits_{k=1}^{2h-1}a_k,\sum\limits_{k=1}^{4h-3}a_k;\sum\limits_{k=1}^{4h-3}a_k-1,\sum\limits_{k=1}^{4h-3}a_k\bigg)\nonumber\\
&=\sum_{x=\sum\limits_{k=1}^{2h-1}a_k+1}^{\sum\limits_{k=1}^{4h-3}a_k}(-1)^{\sum\limits_{k=1}^{4h-3}a_k-x+1}\prod\limits_{i=x}^{\sum\limits_{k=1}^{4h-3}a_k}q^{D}_i\nonumber\\
&=\frac{-q_A(1-b_{4h-3})}{1+q_A}+b_{4h-3}\frac{-q_B(1-b_{4h-4})}{1+q_B}+\cdots +\prod_{i=1}^{2h-3}b_{4h-2-i}\frac{-q_B(1-b_{2h})}{1+q_B}\nonumber\\
&=\frac{-q_A}{1+q_A}+\frac{q_B-q_A}{(1+q_A)(1+q_B)}\bigg(-b_{4h-3}+b_{4h-3}b_{4h-4}-\cdots -\prod_{i=1}^{2h-3}b_{4h-2-i}\bigg)\nonumber\\
&\qquad{}+\frac{q_B}{1+q_B}\prod_{i=1}^{2h-2}b_{4h-2-i}\nonumber\\
&=\frac{-q_A}{1+q_A}+\frac{q_B-q_A}{(1+q_A)(1+q_B)}\bigg(\sum\limits_{j=1}^{2h-2}(-1)^j\prod_{i=1}^{j}b_{4h-2-i}\bigg)+\frac{q_A}{1+q_A}\prod_{i=1}^{2h-2}b_{2h-1+i}
\end{align}
We note that each strategy segment is a proportional sequence sum if the second product term retrieves the strategy in the direction of the arrow. The difference between the two strategies is as follows.
$$D\to A\cdots A\cdots \overset{\sum\limits_{k=1}^{4h-3}a_k}{\mid}\overset{\longrightarrow}{\underbrace{B\cdots B}_{s_{h-1}}}\overset{\longrightarrow}{\underbrace{A\cdots A}_{r_h}}\overset{\sum\limits_{k=1}^{4h-1}a_k}{\mid}\cdots ,$$
$$D'\to A\cdots A\cdots \overset{\sum\limits_{k=1}^{4h-3}a_k}{\mid}\overset{\longrightarrow}{\underbrace{A\cdots A}_{r_h}}\overset{\longrightarrow}{\underbrace{B\cdots B}_{s_{h-1}}}\overset{\sum\limits_{k=1}^{4h-1}a_k}{\mid}\cdots $$
 Then
\begin{align}\label{eq16}
&\Delta\xi\bigg(\sum\limits_{k=1}^{4h-3}a_k,\sum\limits_{k=1}^{4h-3}a_k+1;\sum\limits_{k=1}^{4h-3}a_k,\sum\limits_{k=1}^{4h-1}a_k\bigg)\nonumber\\
&=\sum_{y=\sum\limits_{k=1}^{4h-3}a_k+1}^{\sum\limits_{k=1}^{4h-1}a_k}(-1)^{y-\sum\limits_{k=1}^{4h-3}a_k}\prod\limits_{i=\sum\limits_{k=1}^{4h-3}a_k+1}^{y}(q^{D'}_i-q^{D}_i)\nonumber\\
&=\frac{-q_A(1-b_{4h-1})}{1+q_A}+b_{4h-1}\frac{-q_B(1-b_{4h-2})}{1+q_B}-\frac{-q_B(1-b_{4h-2})}{1+q_B}-b_{4h-2}\frac{-q_A(1-b_{4h-1})}{1+q_A}\nonumber\\
&=\frac{q_B-q_A}{(1+q_A)(1+q_B)}(1-b_{4h-2})(1-b_{4h-1}).
\end{align}
Next, we study the part of the swap segment that contains only $A$s. In fact, strategy $D'$ can also be regarded as the exchange of $A^{r_1}\cdots B^{s_{h-2}}A^{r_{h-1}}$ with $B^{s_{h-1}}$ in strategy $D$ to produce two new strategy exchange parts. As discussed above, we study the product term that does not contain the elements of the swapped part and the product term that contains all the elements of the swap part as follows.
$$D\to A\cdots A\cdots \underbrace{B\cdots B}_{s_{h-1}}\underbrace{A..A\rlap{$\overbrace{\phantom{AAAAAA}}^{product \ \ term}$}A..A}_{r_h}\underbrace{B..BB..B}_{s_h}\overbrace{\underbrace{A\cdots A}_{r_{1}}\cdots \underbrace{A\cdots A}_{r_{h-1}}\underbrace{B\cdots B}_{s_{h-1}}}^{swap\ \ parts}\underbrace{A\cdots A}_{r_{h}}\underbrace{B\cdots B}_{s_{h}},$$
$$D'\to A\cdots A\cdots \underbrace{A..A\rlap{$\overbrace{\phantom{AAAAAA}}^{product \ \ term}$}A..A}_{r_h}\underbrace{B..BB..B}_{s_h}\overbrace{\underbrace{B\cdots B}_{s_{h-1}}\underbrace{A\cdots A}_{r_{1}}\cdots \underbrace{A\cdots A}_{r_{h-1}}}^{swap\ \ parts}\underbrace{A\cdots A}_{r_{h}}\underbrace{B\cdots B}_{s_{h-1}}\underbrace{B\cdots B}_{s_{h}}.$$
Then if $\sum\limits_{k=1}^{2h-2}a_k< x< \sum\limits_{k=1}^{2h-1}a_k$ and $x<y\le r+s$, we have $\prod\limits_{i=x-a_{2h-1}}^{y-a_{2h-1}}q^{D'}_i=\prod\limits_{i=x}^{y}q^{D}_i$.
$D$ and $D'$ can then be written as follows.
$$D\to A\cdots A\cdots \underbrace{B\cdots B}_{s_{h-1}}\underbrace{A..A\rlap{$\overbrace{\phantom{AAAAAAAAAAAAAAAAA}}^{product \ \ term}$}A..A}_{r_h}\underbrace{B\cdots B}_{s_h}\underbrace{A\cdots A\cdots B\cdots B}_{swap\ \ parts}\underbrace{A..AA..A}_{r_h}\underbrace{B\cdots B}_{s_h},$$
$$D'\to A\cdots A\cdots \underbrace{A..A\rlap{$\overbrace{\phantom{AAAAAAAAAAAAAAAAA}}^{product \ \ term}$}A..A}_{r_h}\underbrace{B\cdots B}_{s_h}\underbrace{B\cdots B\cdots A\cdots A}_{swap\ \ parts}\underbrace{A..AA..A}_{r_{h}}\underbrace{B\cdots B}_{s_{h-1}}\underbrace{B\cdots B}_{s_h}.$$
Then if $\sum\limits_{k=1}^{2h-2}a_k< x\le \sum\limits_{k=1}^{2h-1}a_k$ and $\sum\limits_{k=1}^{4h-2}a_k\le y< r+s+x$, we have $\prod\limits_{i=x-a_{2h-2}}^{y-a_{2h-2}}q^{D'}_i=\prod\limits_{i=x}^{y}q^{D}_i$.\\
Then by Equations (\ref{eq12}) and (\ref{eq13}), for the interval $\sum\limits_{k=1}^{2h-2}a_k< x\le \sum\limits_{k=1}^{2h-1}a_k$ of strategy $D$ and for the interval $\sum\limits_{k=1}^{2h-3}a_k< x\le \sum\limits_{k=1}^{2h-3}a_k+a_{2h-1}$ of strategy $D'$, we have
\begin{align}\label{eq17}
&\xi^{D'}\bigg(\sum\limits_{k=1}^{2h-3}a_k,\sum\limits_{k=1}^{2h-3}a_k+a_{2h-1};x,r+s+x-1\bigg)-\xi^{D}\bigg(\sum\limits_{k=1}^{2h-2}a_k,\sum\limits_{k=1}^{2h-1}a_k;x,r+s+x-1\bigg)\nonumber\\
&=\sum_{x=\sum\limits_{k=1}^{2h-2}a_k+1}^{\sum\limits_{k=1}^{2h-1}a_k}\sum_{y=x+1}^{r+s+x-1}(-1)^{y-x+1}(\prod\limits_{i=x-a_{2h-2}}^{y-a_{2h-2}}q^{D'}_i-\prod\limits_{i=x}^{y}q^{D}_i)\nonumber\\
&=\xi^D\bigg(\sum\limits_{k=1}^{2h-2}a_k,\sum\limits_{k=1}^{2h-1}a_k;\sum\limits_{k=1}^{2h}a_k-1,\sum\limits_{k=1}^{2h}a_k\bigg)\nonumber\\
&\qquad{}\times\bigg(\sum_{y=\sum\limits_{k=1}^{2h}a_k+1}^{\sum\limits_{k=1}^{4h-2}a_k}(-1)^{y-\sum\limits_{k=1}^{2h}a_k}\prod\limits_{i=\sum\limits_{k=1}^{2h}a_k+1}^{y}(q^{D'}_{i-a_{2h-2}}-q^{D}_i)\bigg).
\end{align}
Similarly, we first study the permutation of the proportional sequence of the first product term.
$$D\to A\cdots A\cdots \overset{\sum\limits_{k=1}^{2h-2}a_k}{\mid}\overset{\longleftarrow}{\underbrace{A\cdots A}_{r_h}}\underbrace{B\cdots B}_{s_h}\overset{\sum\limits_{k=1}^{2h}a_k}{\mid}\cdots $$
 Then
\begin{equation}\label{eq18}
\begin{aligned}
\xi^D\bigg(\sum\limits_{k=1}^{2h-2}a_k,\sum\limits_{k=1}^{2h-1}a_k;\sum\limits_{k=1}^{2h}a_k-1,\sum\limits_{k=1}^{2h}a_k\bigg)=b_{2h}\frac{-q_A(1-b_{2h-1})}{1+q_A}.
\end{aligned}
\end{equation}
We next study the permutation of the proportional sequence of the second product term.
$$D\to A\cdots A\cdots \overset{\sum\limits_{k=1}^{2h}a_k}{\mid}\overset{\longrightarrow}{\underbrace{A\cdots A}_{r_1}}\overset{\longrightarrow}{\underbrace{B\cdots B}_{s_1}}\cdots \overset{\longrightarrow}{\underbrace{A\cdots A}_{r_{h-1}}}\overset{\longrightarrow}{\underbrace{B\cdots B}_{s_{h-1}}}\overset{\sum\limits_{k=1}^{4h-2}a_k}{\mid}\cdots ,$$
$$D'\to A\cdots A\cdots \overset{\sum\limits_{k=1}^{2h}a_k-a_{2h-2}}{\mid}\overset{\longrightarrow}{\underbrace{B\cdots B}_{s_{h-1}}}\overset{\longrightarrow}{\underbrace{A\cdots A}_{r_1}}\cdots \overset{\longrightarrow}{\underbrace{B\cdots B}_{s_{h-2}}}\overset{\longrightarrow}{\underbrace{A\cdots A}_{r_{h-1}}}\overset{\sum\limits_{k=1}^{4h-2}a_k-a_{2h-2}}{\mid}\cdots $$
 Then
\begin{align}\label{eq19}
&\sum_{y=\sum\limits_{k=1}^{2h}a_k+1}^{\sum\limits_{k=1}^{4h-2}a_k}(-1)^{y-\sum\limits_{k=1}^{2h}a_k}\prod\limits_{i=\sum\limits_{k=1}^{2h}a_k+1}^{y}(q^{D'}_{i-a_{2h-2}}-q^{D}_i)\nonumber\\
&=\frac{-q_B}{1+q_B}+\frac{q_A-q_B}{(1+q_A)(1+q_B)}\bigg(\sum\limits_{j=0}^{2h-4}(-1)^{j+1}b_{2h-2}\prod_{i=1}^{j}b_{2h+i}\bigg)+\frac{q_A}{1+q_A}\prod_{i=1}^{2h-2}b_{2h+i}\nonumber\\
&\qquad{}-\frac{-q_A}{1+q_A}-\frac{q_B-q_A}{(1+q_A)(1+q_B)}\bigg(\sum\limits_{j=1}^{2h-3}(-1)^j\prod_{i=1}^{j}b_{2h+i}\bigg)-\frac{q_B}{1+q_B}\prod_{i=1}^{2h-2}b_{2h+i}\nonumber\\
&=\frac{q_B-q_A}{(1+q_A)(1+q_B)}(1-b_{2h-2})\bigg(\sum\limits_{j=1}^{2h-2}(-1)^j\prod_{i=2}^{j}b_{2h-1+i}\bigg).
\end{align}

Finally, we observe the part of the original swap segment that contains only $B$s. Similarly, strategy $D'$ can be regarded as the exchange of $B^{s_{h}}A^{r_{1}}\cdots B^{s_{h-2}}$ with $A^{r_{h}}$ in strategy $D$ to yield two new strategy exchange parts as follows.
$$D\to A\cdots A\cdots \underbrace{\rlap{$\overbrace{\phantom{A}}^{product\ \ term}$}B..B..B..B}_{s_{h-1}}\overbrace{\underbrace{A\cdots A}_{r_h}\underbrace{B\cdots B}_{s_h}\cdots \underbrace{B\cdots B}_{s_{h-2}}}^{swap\ \ parts}\underbrace{A\cdots A}_{r_{h-1}}\underbrace{B\cdots B}_{s_{h-1}}\underbrace{A\cdots A}_{r_{h}}\underbrace{B\cdots B}_{s_{h}},$$
$$D'\to A\cdots A\cdots \underbrace{A\cdots A}_{r_{h}}\underbrace{\rlap{$\overbrace{\phantom{A}}^{product\ \ term}$}B..B..B..B}_{s_{h-1}}\overbrace{\underbrace{B\cdots B}_{s_h}\cdots \underbrace{B\cdots B}_{s_{h-2}}\underbrace{A\cdots A}_{r_h}}^{swap\ \ parts}\underbrace{A\cdots A}_{r_{h-1}}\underbrace{B\cdots B}_{s_{h-1}}\underbrace{B\cdots B}_{s_{h}}.$$
Then if $\sum\limits_{k=1}^{2h-3}a_k< x< \sum\limits_{k=1}^{2h-2}a_k$ and $x<y\le \sum\limits_{k=1}^{2h-2}a_k$, we have $\prod\limits_{i=x+a_{2h-1}}^{y+a_{2h-1}}q^{D'}_i=\prod\limits_{i=x}^{y}q^{D}_i$.\\
$D$ and $D'$ can then be written as follows.
$$D\to A\cdots A\cdots \underbrace{B..B\rlap{$\overbrace{\phantom{AAAAAAAAAAAAAA}}^{product\ \ term}$}B..B}_{s_{h-1}}\underbrace{A\cdots A\cdots B\cdots B}_{swap\ \ parts}\underbrace{A..AA..A}_{r_{h-1}}\underbrace{B\cdots B}_{s_{h-1}}\underbrace{A\cdots A}_{r_h}\underbrace{B\cdots B}_{s_{h-1}},$$
$$D'\to A\cdots A\cdots \underbrace{A\cdots A}_{r_{h}}\underbrace{B..B\rlap{$\overbrace{\phantom{AAAAAAAAAAAAAA}}^{product\ \ term}$}B..B}_{s_{h-1}}\underbrace{B\cdots B\cdots A\cdots A}_{swap\ \ parts}\underbrace{A..AA..A}_{r_{h-1}}\underbrace{B\cdots B}_{s_{h-1}}\underbrace{B\cdots B}_{s_{h}}.$$
Then if $\sum\limits_{k=1}^{2h-3}a_k< x\le \sum\limits_{k=1}^{2h-2}a_k$ and $\sum\limits_{k=1}^{4h-4}a_k<y<r+s+x$, we have $\prod\limits_{i=x+a_{2h-1}}^{y+a_{2h-1}}q^{D’}_i=\prod\limits_{i=x}^{y}q^{D}_i$.
Similarly, by Equations (\ref{eq12}) and (\ref{eq13}),  for the interval $\sum\limits_{k=1}^{2h-3}a_k< x\le \sum\limits_{k=1}^{2h-2}a_k$ of strategy $D$ and for the interval $\sum\limits_{k=1}^{2h-3}a_k+a_{2h-1}< x\le \sum\limits_{k=1}^{2h-1}a_k$ of strategy $D'$, we have
\begin{align}\label{eq20}
&\xi^{D'}\bigg(\sum\limits_{k=1}^{2h-3}a_k+a_{2h-1},\sum\limits_{k=1}^{2h-1}a_k;x,r+s+x-1\bigg)-\xi^{D}\bigg(\sum\limits_{k=1}^{2h-3}a_k,\sum\limits_{k=1}^{2h-2}a_k;x,r+s+x-1\bigg)\nonumber\\
&=\sum_{x=\sum\limits_{k=1}^{2h-3}a_k+1}^{\sum\limits_{k=1}^{2h-2}a_k}\sum_{y=x+1}^{r+s+x-1}(-1)^{y-x+1}(\prod\limits_{i=x+a_{2h-1}}^{y+a_{2h-1}}q^{D'}_i-\prod\limits_{i=x}^{y}q^{D}_i)\nonumber\\
&=\xi^{D}\bigg(\sum\limits_{k=1}^{2h-3}a_k,\sum\limits_{k=1}^{2h-2}a_k;\sum\limits_{k=1}^{2h-2}a_k-1,\sum\limits_{k=1}^{2h-2}a_k\bigg)\nonumber\\
&\qquad{}\times\bigg(\sum_{y=\sum\limits_{k=1}^{2h-2}a_k+1}^{\sum\limits_{k=1}^{4h-4}a_k}(-1)^{y-\sum\limits_{k=1}^{2h-2}a_k}\prod\limits_{i=\sum\limits_{k=1}^{2h-2}a_k+1}^{y}(q^{D'}_{i+a_{2h-1}}-q^{D}_i)\bigg).
\end{align}
Similarly, the permutation of the proportional sequence of the first product term is as follows.
$$D\to A\cdots A\cdots \overset{\sum\limits_{k=1}^{2h-3}a_k}{\mid}\overset{\longleftarrow}{\underbrace{B\cdots B}_{s_{h-1}}}\overset{\sum\limits_{k=1}^{2h-2}a_k}{\mid}\cdots $$
Then
\begin{equation}\label{eq21}
\begin{aligned}
\xi^{D}\bigg(\sum\limits_{k=1}^{2h-3}a_k,\sum\limits_{k=1}^{2h-2}a_k;\sum\limits_{k=1}^{2h-2}a_k-1,\sum\limits_{k=1}^{2h-2}a_k\bigg)=\frac{-q_B(1-b_{2h-2})}{1+q_B}.
\end{aligned}
\end{equation}
The permutation of the proportional sequence of the second product term is as follows.
$$D\to A\cdots A\cdots \overset{\sum\limits_{k=1}^{2h-2}a_k}{\mid}\overset{\longrightarrow}{\underbrace{A\cdots A}_{r_h}}\overset{\longrightarrow}{\underbrace{B\cdots B}_{s_h}}\cdots \overset{\longrightarrow}{\underbrace{A\cdots A}_{r_{h-2}}}\overset{\longrightarrow}{\underbrace{B\cdots B}_{s_{h-2}}}\overset{\sum\limits_{k=1}^{4h-4}a_k}{\mid}\cdots ,$$
$$D'\to A\cdots A\cdots \overset{\sum\limits_{k=1}^{2h-2}a_k+a_{2h-1}}{\mid}\overset{\longrightarrow}{\underbrace{B\cdots B}_{s_{h}}}\overset{\longrightarrow}{\underbrace{A\cdots A}_{r_1}}\cdots \overset{\longrightarrow}{\underbrace{B\cdots B}_{s_{h-2}}}\overset{\longrightarrow}{\underbrace{A\cdots A}_{r_{h}}}\overset{\sum\limits_{k=1}^{4h-4}a_k+a_{2h-1}}{\mid}\cdots $$
Then
\begin{align}\label{eq22}
&\sum_{y=\sum\limits_{k=1}^{2h-2}a_k+1}^{\sum\limits_{k=1}^{4h-4}a_k}(-1)^{y-\sum\limits_{k=1}^{2h-2}a_k}\prod\limits_{i=\sum\limits_{k=1}^{2h-2}a_k+1}^{y}(q^{D'}_{i+a_{2h-1}}-q^{D}_i)\nonumber\\
&=\frac{-q_B}{1+q_B}+\frac{q_A-q_B}{(1+q_A)(1+q_B)}\bigg(\sum\limits_{j=1}^{2h-3}(-1)^{j}\prod_{i=1}^{j}b_{2h-1+i}\bigg)+\frac{q_A}{1+q_A}\prod_{i=1}^{2h-2}b_{2h-2+i}\nonumber\\
&\qquad{}-\frac{-q_A}{1+q_A}-\frac{q_B-q_A}{(1+q_A)(1+q_B)}\bigg(\sum\limits_{j=1}^{2h-3}(-1)^j\prod_{i=1}^{j}b_{2h-2+i}\bigg)-\frac{q_B}{1+q_B}\prod_{i=1}^{2h-2}b_{2h-2+i}\nonumber\\
&=\frac{q_A-q_B}{(1+q_A)(1+q_B)}(1-b_{2h-1})\bigg(\sum\limits_{j=0}^{2h-3}(-1)^{j}\prod_{i=1}^{j}b_{2h-1+i}\bigg).
\end{align}

Hence by Lemma \ref{le3} and Equations (\ref{eq13})--(\ref{eq22}), we have
\begin{align*}
R_D-R_{D'}&=2S(1-{b_{2h-2}})(1-{b_{2h-1}})\bigg(\sum\limits_{j=0}^{2h-3}(-1)^{j}\prod\limits_{i=1}^{j}{b_{2h-1+i}}+\sum\limits_{j=1}^{2h-2}(-1)^{j}\prod\limits_{i=1}^{j}{b_{4h-2-i}}\bigg)\\
&=2S(1-{b_{2h-2}})(1-{b_{2h-1}})\bigg(\sum\limits_{j=0}^{2h-3}(-1)^{j}(\prod\limits_{i=0}^{j-1}{b_{i}}+\prod\limits_{i=j}^{2h-3}{b_i})\bigg),
\end{align*}
where $\boldsymbol{a}=(r_1,s_1,r_2,s_2,\cdots ,r_h,s_h)$ and $$S=\frac{(q_A-q_B)^2(1+(-1)^{r+s}q_A^rq_B^s)}{(r+s)(1+q_A)^2(1+q_B)^2(1-(q_A^rq_B^s)^2)}.$$
\end{proof}

\section*{Appendix B}

\begin{proof} [Proof of Theorem \ref{th2}]
We first exchange strategy segment $B$ of the $(2h-2)$-th segment with strategy segment $A$ of the $(2h-1)$-th segment to obtain a new strategy $D_{h-1}$. Then we exchange strategy segment $B$ of the $(2h-4)$-th segment with strategy segment $A$ of the $(2h-3)$-th segment in $D_{h-1}$ to obtain a new strategy $D_{h-2}$. If we keep swapping in this way, we obtain the strategy
$$D_{l}:=A^{r_1}B^{s_1}\cdots A^{r_{l-1}}B^{s_{l-1}}A^{r_l+r_{l+1}+\cdots +r_{h}}B^{s_l+s_{l+1}+\cdots +s_{h}},$$
where $\boldsymbol{a(l)}:=\bigg(r_1,s_1,r_2,s_2,\cdots ,r_{l-1},s_{l-1},\sum\limits_{m=l}^{h}r_{m},\sum\limits_{m=l}^{h}s_{m}\bigg)$. Every time we exchange $B^{a_{2l-2}}$ and $A^{a_{2l-1}}$, we obtain a new $D_{l-1}$ and $\boldsymbol{a(l-1)}$, and the asymptotic profit per coup difference corresponding to the exchange can be calculated by Lemma \ref{le5} for each exchange. Thus, we can continue to exchange and add all the differences to complete the proof.  Next we show that for $1\le l \le h$, $R_{D_l}=2Q_lS$, where
$$Q_l=l+\sum\limits_{m=1}^{2l}\sum\limits_{j=1}^{2l-1}(-1)^j\prod\limits_{i=m}^{m+j-1}{b_i(l)}+l\prod\limits_{i=1}^{2l}{b_i(l)}.$$
Given Lemma \ref{le4} and Lemma \ref{le5}, we only need to verify the form of $Q_l$, and we prove it by induction.
By Lemma \ref{le4}, if $l=1$, then the conclusion is trivial. Now, we suppose that for $l = n-1$, $1<n\le h$, the following equation holds.
\begin{align}\label{eq23}
Q_{n-1}=(n-1)+\sum\limits_{m=1}^{2n-2}\sum\limits_{j=1}^{2n-3}(-1)^j\prod\limits_{i=m}^{m+j-1}{b_i(n-1)}+(n-1)\prod\limits_{i=1}^{2n-2}{b_i(n-1)}
\end{align}
We proceed to show that it also holds for $l=n$.\\
\begin{align*}
D_{n}&=A^{r_1}B^{s_1}\cdots A^{r_{n-1}}B^{s_{n-1}}A^{\sum\limits_{m=n}^{h}r_{m}}B^{\sum\limits_{m=n}^{h}s_{m}},\\ D_{n-1}&=A^{r_1}B^{s_1}\cdots A^{r_{n-1}}A^{\sum\limits_{m=n}^{h}r_{m}}B^{s_{n-1}}B^{\sum\limits_{m=n}^{h}s_{m}}.
\end{align*}
We observe the relationship of the element's function between $b_i(n-1)$ and $b_i(n)$ in the parameter vectors $\boldsymbol{a(n-1)}$ and $\boldsymbol{a(n)}$ of the two strategies after the exchange as follows.
\begin{enumerate}
\item $b_i(n-1)=b_i(n)$, for $1\le i \le 2n-4$.\\
\item $b_i(n-1)=b_{i+2}(n)$, for $2n-1\le i \le 4n-6$.\\
\item ${b_{2n-3}(n-1)}={b_{2n-3}(n)}{b_{2n-1}(n)}$.\\
\item ${b_{2n-2}(n-1)}={b_{2n-2}(n)}b_{2n}(n)$.
\end{enumerate}
Let us first look at the second summation in Equation (\ref{eq23}) and separate the parts in which $b_i(n-1)$ and $b_i(n)$ are equal as follows.
\begin{align}\label{eq24}
&\sum\limits_{m=1}^{2n-2}\sum\limits_{j=1}^{2n-3}(-1)^j\prod\limits_{i=m}^{m+j-1}{b_i(n-1)}=\sum\limits_{m=1}^{2n-2}\sum\limits_{k=m}^{2n-4+m}(-1)^{k-m+1}\prod\limits_{i=m}^{k}{b_i(n-1)}\nonumber\\
&=\bigg(\sum\limits_{m=1}^{2n-4}+\sum\limits_{m=2n-3}^{2n-2}\bigg)\bigg(\sum\limits_{k=m}^{2n-4}+\sum\limits_{k=2n-3}^{2n-2}+\sum\limits_{k=2n-1}^{2n-4+m}\bigg)(-1)^{k-m+1}\prod\limits_{i=m}^{k}{b_i(n-1)}\nonumber\\
&=\sum\limits_{m=1}^{2n-4}\bigg(\sum\limits_{k=m}^{2n-4}+\sum\limits_{k=2n+1}^{2n-2+m}\bigg)(-1)^{k-m+1}\prod\limits_{i=m}^{k}{b_i(n)}\nonumber\\
&\qquad{}+\sum\limits_{m=1}^{2n-4}(-1)^m\bigg(b_{2n-1}(n)\prod\limits_{i=m}^{2n-3}{b_i(n)}-\prod\limits_{i=m}^{2n}{b_i(n)}\bigg)\nonumber\\
&\qquad{}-b_{2n-3}(n)b_{2n-1}(n)-b_{2n-2}(n)b_{2n}(n)+b_{2n-3}(n)b_{2n-2}(n)b_{2n-1}(n)b_{2n}(n)\nonumber\\
&\qquad{}+\sum\limits_{k=2n+1}^{4n-5}(-1)^{k}\prod\limits_{i=2n-3}^{k}{b_i(n)}+\sum\limits_{k=2n+1}^{4n-4}(-1)^{k-1}b_{2n-2}(n)\prod\limits_{i=2n}^{k}{b_i(n)}\nonumber\\
&=\sum\limits_{m=1}^{2n-3}\bigg(\sum\limits_{k=m}^{2n-4}+\sum\limits_{k=2n}^{2n-2+m}\bigg)(-1)^{k-m+1}\prod\limits_{i=m}^{k}{b_i(n)}+\sum\limits_{m=1}^{2n-4}(-1)^mb_{2n-1}(n)\prod\limits_{i=m}^{2n-3}{b_i(n)}\nonumber\\
&\qquad{}-b_{2n-3}(n)b_{2n-1}(n)-b_{2n-2}(n)b_{2n}(n)+\sum\limits_{k=1}^{2n-4}(-1)^{k-1}b_{2n-2}(n)\prod\limits_{i=0}^{k}{b_i(n)}.
\end{align}
We find that after separating the unequal parts of $b_i(n-1)$ and $b_i(n)$, the partially separated sums can be merged back, since $\prod\limits_{i=2n-3}^{2n-2}b_i(n-1)=\prod\limits_{i=2n-3}^{2n}b_i(n)$, leading to the second equation above. Now by Lemma \ref{le5}, we have
\begin{align*}
&Q_{n}-Q_{n-1}=(1-{b_{2n-2}}(n))(1-{b_{2n-1}}(n))\bigg(\sum\limits_{j=0}^{2n-3}(-1)^{j}\bigg(\prod\limits_{i=0}^{j-1}{b_{i}(n)}+\prod\limits_{i=j}^{2n-3}{b_i}(n)\bigg)\bigg)\\
&=(1-{b_{2n-2}}(n))(1-{b_{2n-1}}(n))\\
&\qquad{}\times\bigg(\sum\limits_{j=1}^{2n-4}(-1)^{j-1}\prod\limits_{i=0}^{j}{b_{i}(n)}+1-b_{2n}(n)+\sum\limits_{j=1}^{2n-4}(-1)^j\prod\limits_{i=j}^{2n-3}{b_i}(n)-b_{2n-3}(n)+\prod\limits_{i=0}^{2n-3}b_{i}(n)\bigg)\\
&=\sum\limits_{j=1}^{2n-4}(-1)^{j+1}\bigg(\prod\limits_{i=0}^{j}-\prod\limits_{i=-1}^{j}+\prod\limits_{i=-2}^{j}\bigg)b_i(n)-\sum\limits_{j=1}^{2n-4}(-1)^{j+1}b_{2n-2}(n)\prod\limits_{i=0}^{j}{b_{i}(n)}\\
&\qquad{}+\sum\limits_{j=1}^{2n-4}(-1)^{2n-2-j}\bigg(\prod\limits_{i=j}^{2n-3}-\prod\limits_{i=j}^{2n-2}+\prod\limits_{i=j}^{2n-1}\bigg)b_i(n)-\sum\limits_{j=1}^{2n-4}(-1)^jb_{2n-1}(n)\prod\limits_{i=j}^{2n-3}{b_i}(n)\\
&\qquad{}+\bigg(\prod\limits_{i=2n-2}^{2n-1}-\prod\limits_{i=2n-2}^{2n-2}-\prod\limits_{i=2n-1}^{2n-1}-\prod\limits_{i=2n}^{2n}+\prod\limits_{i=2n-1}^{2n}-\prod\limits_{i=2n-2}^{2n}-\prod\limits_{i=2n-3}^{2n-3}+\prod\limits_{i=2n-3}^{2n-2}-\prod\limits_{i=2n-3}^{2n-1}\bigg)b_i(n)\\
&\qquad{}+\bigg(\prod\limits_{i=0}^{2n-3}-\prod\limits_{i=0}^{2n-2}-\prod\limits_{i=-1}^{2n-3}\bigg)b_i(n)+1+b_{2n-2}(n)b_{2n}(n)+b_{2n-3}(n)b_{2n-1}(n)+\prod\limits_{i=1}^{2n}b_i(n)\\
&=\bigg(\sum\limits_{m=2n-2}^{2n}\sum\limits_{k=m}^{2n-2+m}+\sum\limits_{m=1}^{2n-3}\sum\limits_{k=2n-3}^{2n-1}\bigg)(-1)^{k-m+1}\prod\limits_{i=m}^{k}{b_i(n)}-\sum\limits_{k=1}^{2n-4}(-1)^{k-1}b_{2n-2}(n)\prod\limits_{i=0}^{k}{b_i(n)}\\
&\qquad{}-\sum\limits_{m=1}^{2n-4}(-1)^mb_{2n-1}(n)\prod\limits_{i=m}^{2n-3}{b_i(n)}+1+b_{2n-2}(n)b_{2n}(n)+b_{2n-3}(n)b_{2n-1}(n)+\prod\limits_{i=1}^{2n}b_i(n).\\
\end{align*}
We add the above equation to Equation (\ref{eq24}) and substitute the sum into Equation (\ref{eq23}) to obtain
\begin{align*}
Q_{n}&=n+\sum\limits_{m=1}^{2n}\sum\limits_{k=m}^{2n-2+m}(-1)^{k-m+1}\prod\limits_{i=m}^{k}{b_i(n)}+n\prod\limits_{i=1}^{2n}{b_i(n)}\\
&=n+\sum\limits_{m=1}^{2n}\sum\limits_{j=1}^{2n-1}(-1)^j\prod\limits_{i=m}^{m+j-1}{b_i(n)}+n\prod\limits_{i=1}^{2n}{b_i(n)}
\end{align*}
In particular, if $l=h$, we obtain the exact form of the casino's asymptotic profit per coup, namely,
$$R_D=2QS,$$
where
 $$Q=h+\sum\limits_{m=1}^{2h}\sum\limits_{j=1}^{2h-1}(-1)^j\prod\limits_{i=m}^{m+j-1}{b_i}+h\prod\limits_{i=1}^{2h}{b_i},$$
 $$S=\frac{(q_A-q_B)^2(1+(-1)^{r+s}q_A^rq_B^s)}{(r+s)(1+q_A)^2(1+q_B)^2(1-(q_A^rq_B^s)^2)}.$$
  \end{proof}

\section*{Appendix C}
   In this section, we first analyze the relationship between the sum terms in $Q$ and simplify some of the sum terms into a polynomial. Then, we find the relationship between these polynomials, and finally, we prove the conjecture.

According to our discussion in the previous section, the sign of the casino's asymptotic profit per coup $R_D$ is consistent with the sign of the function $Q$. We can think of $b_1$, $b_2$,$\cdots$,$b_{2h}$ as $2h$ points and draw them in a cycle with the points numbered in the order of $\{b_k\}$'s numbering.
In the figure below, we take $h=5$ as an example.
$$
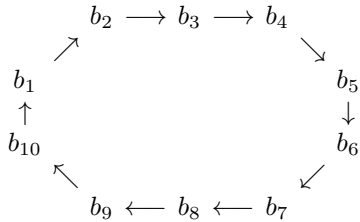
\begin{matrix}
  \  & \  & b_2 & \longrightarrow & b_3 & \longrightarrow & b_4 & \  & \  \\
  \  & \nearrow & \  & \  & \  & \ & \ & \searrow & \  \\
  b_1 & \  & \  & \  & \  & \  & \ & \ & b_5 \\
  \uparrow & \  & \  & \  & \  & \  & \ & \ & \downarrow\\
   b_{10} & \  & \  & \  & \  & \  & \ & \ & b_6 \\
  \  & \nwarrow & \  & \  & \  & \ & \ & \swarrow & \  \\
  \  & \  & b_9 & \longleftarrow & b_8 &\longleftarrow & b_7 & \  & \
\end{matrix}
$$
Returning to the function $Q$, we can rewrite it as the sum of specific product terms as follows:
\begin{equation}\label{eq25}
\begin{aligned}
Q=\sum\limits_{l=1}^{h}\sum\limits_{j=0}^{2h-1}(-1)^{j}\prod\limits_{k=2l+1}^{2l+j}{b_k}+\sum\limits_{l=1}^{h}\sum\limits_{j=0}^{2h-1}(-1)^{j+1}\prod\limits_{k=2l}^{2l+j}{b_k}.\\
\end{aligned}
\end{equation}
In fact, we only need to study the sum of the product term $(-1)^{j-i+1}\prod\limits_{k=i}^{j}{b_k}$ for some $i,j$. Here, $i$ represents the ``starting" point, $j$ refers to the ``end" point, and the product term ``length" is $j-i+1$. We see that $Q$ is the sum of two parts. The first part adds all product terms from length $0$ to $2h-1$ starting from odd ordinal numbers, and the second part adds all product terms from length $1$ to $2h$ starting from even ordinal numbers.

Examining the simplest example first, we assume $b_k>0$ for $1\le k \le 2h$.
$$Q=\sum\limits_{l=1}^{h}(1-{b_{2l}})(1-b_{2l+1}(1-b_{2l+2}(\cdots (1-b_{2l+2h-1})\cdots )))>0.$$
From this heuristic example, we note that for a product term starting from any starting point, we can always find a corresponding term starting from the same starting point whose product ends adjacent to the original end point. We add these two terms to obtain a new polynomial, for example
$$(-1)^{j-i+1}\prod\limits_{k=i}^{j}{b_k}+(-1)^{j-i+2}\prod\limits_{k=i}^{j+1}{b_k}=(-1)^{j-i+1}\prod\limits_{k=i}^{j}{b_k}(1-b_{j+1}).$$
At this stage, the sign of this polynomial is related solely to the product coefficient at the beginning of the polynomial. We can also find the product term corresponding to the adjacent new end point and add these three terms to form a new polynomial.
$$\sum\limits_{l=1}^{3}(-1)^{j-i+l}\prod\limits_{k=i}^{j+l-1}{b_k}=(-1)^{j-i+1}\prod\limits_{k=i}^{j}{b_k}(1-b_{j+1}(1-b_{j+2})).$$
At this stage, we can define the first type of iteration as that starting from the product term corresponding to the fixed starting point and iterating from different ending points to the starting point. Taking the starting point as $1$ and $h=4$ as an example, we consider the sum of the product terms of lengths $0$ to $2h-1$ according to Equation (\ref{eq25}). The iterative architecture and the resulting polynomial are as follows.
$$1\leftarrow b_1\leftarrow b_2\leftarrow b_3\leftarrow b_4\leftarrow b_5\leftarrow b_6\leftarrow b_7,$$
$$\sum\limits_{l=0}^{7}(-1)^{l}\prod\limits_{k=1}^{l}{b_k}=1-b_1(1-b_2(1-b_3(1-b_4(1-b_5(1-b_6(1-b_7)))))).$$
Similarly, we can also fix the end point of the product term and add several terms adjacent to the starting point to obtain a polynomial.
$$\sum\limits_{l=1}^{3}(-1)^{j-i+l}\prod\limits_{k=i-l+1}^{j}{b_k}=(-1)^{j-i+1}\prod\limits_{k=i}^{j}{b_k}(1-b_{i-1}(1-b_{i-2})).$$
In the same way, we can also define a second type of iteration as iterating from different starting points to the end point, starting with the product term corresponding to the fixed end point. Taking the end point as $8$ and $h=4$ as an example, we consider the sum of the product terms of lengths $1$ to $2h$  according to Equation (\ref{eq25}). Then, the iterative architecture and the resulting polynomial are as follows.
$$b_1\rightarrow b_2\rightarrow b_3\rightarrow b_4\rightarrow b_5\rightarrow b_6\rightarrow b_7\rightarrow -b_8,$$
$$\sum\limits_{l=1}^{8}(-1)^{l}\prod\limits_{k=9-l}^{8}{b_k}=-b_8(1-b_7(1-b_6(1-b_5(1-b_4(1-b_3(1-b_2(1-b_1))))))).$$
 In general, however, if $b_k<0$, we cannot iterate this way forever. Because $1-b_k>1$, we cannot always ensure a $b_k$ value between $0$ and $1$. We call these $k$ such that $b_k<0$  ``negative" points and the corresponding $k$ such that $b_k>0$  ``positive" points, and we can number the negative points as follows.

\begin{definition}\label{de2}
We consider the general sequence $\{b_k\}$ with a period of $2h$, where $0<\\| b_k\\|<1$, for $1\le k \le 2h$ with all $b_i<0$ renumbered as $c_k$. We call the points of these points ``negative" points; the rest are called ``positive" points.
\begin{enumerate}
\item $c_1:=\min\{i\mid i>0,\ {b_{i}}<0\}$,\\
\item $c_k=\min\{i\mid i>c_{k-1},\ {b_{i}}<0\}$,\\
\item $c_{\delta}:=\max\{c_k\mid c_k\le 2h\}$.
\end{enumerate}
Then, $\delta$ represents the number of negative points in a periodic sequence. Accordingly, we also extend the definition of $c_k$ as $c_k=c_{k\pm \delta}$.
\end{definition}
If $\delta=1$, we can still follow the simplest example. In fact, without losing generality, we can suppose that ${b_{m^*}}<0$ and $m^*$ are odd. Then
\begin{align*}
Q&=\sum\limits_{l=1}^{h}(1-{b_{2l}})\bigg((1-b_{2l+1}(\cdots (1-b_{m^*-2}(1-b_{m^*-1}))\cdots ))\\
&\qquad{}+\lvert\mu(2l+1,m^*)\rvert(1-b_{m^*+1}(\cdots (1-b_{2h+2l-2}(1-b_{2h+2l-1}))\cdots ))\bigg)>0.
\end{align*}

In the subsequent discussion, we only consider the case of $\delta>1$; for convenience, we introduce the following functions.
\begin{definition}\label{de3}
 For $i,j\in \mathbb{Z}$, we define the product term function $\mu(i,j)$ as
 $$
\mu(i,j):=
\begin{cases}
(-1)^{j-i+1}\prod\limits_{k=i}^{j}{b_k}& \text{ $i\le j$ } \\
1 & \text{ $i>j$ },
\end{cases}
$$
we call $i$ the ``starting" point of the product term, $j$ the ``end" point of the product term, and $j-i+1$ the ``length" of the product term.

In particular, if the starting and end points are negative points, we have another way of expressing the product term. For $\delta>1$, $0<j<\delta$, we define a function $\nu(i,j)$ as
 $$
\nu(i,j):=
\begin{cases}
(-1)^{c_{j}-c_{i}+1}\prod\limits_{k=c_{i}}^{c_{j}}{b_k}& \text{ $i\le j$} \\
1 & \text{ $i>j$ },
\end{cases}
$$
and we denote $\tilde{\nu}(i,j):=\mu(c_i+1,c_{j})$.

We define a polynomial route to represent the polynomial after our iteration. We use the direction of the arrow to indicate the direction of iteration, corresponding to the following forms of the polynomial.
  \begin{enumerate}
\item The first type of iterative direction
$$\mu(i,j)\leftarrow b_{j+1} \leftarrow b_{j+2} \leftarrow\cdots  \leftarrow b_n,$$
$$\mu(i,j)(1-b_{j+1}(1-b_{j+2}(\cdots (1-b_{n-1}(1-b_n))\cdots ))).$$
\item The second type of iterative direction
$$b_m\rightarrow\cdots \rightarrow b_{i-2}\rightarrow b_{i-1}\rightarrow\mu(i,j),$$
$$\mu(i,j)(1-b_{i-1}(1-b_{i-2}(\cdots (1-b_{m+1}(1-b_m))\cdots ))).$$
\item The mixed iterative direction
$$b_m\rightarrow\cdots \rightarrow b_{i-2}\rightarrow b_{i-1}\rightarrow\mu(i,j)\leftarrow b_{j+1} \leftarrow b_{j+2} \leftarrow\cdots  \leftarrow b_n,$$
$$\mu(i,j)(1-b_{j+1}(\cdots (1-b_{n-1}(1-b_n))\cdots ))(1-b_{i-1}(\cdots (1-b_{m+1}(1-b_m))\cdots )).$$
  \end{enumerate}

\end{definition}
\begin{remark}
According to our previous discussion, $\mu(i,j)$ is the product term with $i$ as the starting point and $j$ as the end point, and $\nu(i,j)$ is the product term with $c_i$ as the starting point and $c_j$ as the end point. Thus, the function $Q$ can also be abbreviated as follows.
\begin{equation}\label{eq26}
Q=\sum\limits_{l=1}^{h}\sum\limits_{j=0}^{2h-1}\mu(2l+1,2l+j)+\sum\limits_{l=1}^{h}\sum\limits_{j=0}^{2h-1}\mu(2l,2l+j).\\
\end{equation}
Due to the nature of positive and negative points, the following relationship between functions $\mu$, $\nu$, and $\tilde{\nu}$ holds.
\end{remark}

\begin{lemma}\label{pr4}
  From Definitions \ref{de2} and \ref{de3}, the following conclusions can be inferred.
  \begin{enumerate}
\item $\mu(c_i,c_{j})=\nu(i,j)=(-1)^{c_{j}-c_{i}+j-i}\prod\limits_{k=c_{i}}^{c_{j}}{\lvert b_k\rvert}$ \ \ for $i\le j$.\\
\item $\mu(c_i,c_{j}-1)$, $\nu(i,j)$ and $\tilde{\nu}(i,j)$ have the same sign.\\
\item If $0<b_k<1$ \ \ for $m\le k\le n$, then the sign of the iterative polynomial defined in \ref{de3} is the same as that of $\mu(i,j)$.\\
  \end{enumerate}
\end{lemma}

If we start from any starting point and end at a positive point to form a product term, we can always add the product terms of its adjacent end points to form a polynomial. However, in the process of iteration, when the end point of the product term is a negative point, the iteration stops. Next, we try to find a factorization rule so that our sum can skip some negative points so as to allow iteration to continue. By requiring that the product term decomposed by this rule be unique and not interfere with any other, we obtain the following lemma.

\begin{lemma}\label{le6}
For $\delta>1$, $i<j<i+\delta$, if $v(i,j):=(-1)^{c_{j}-c_i+j-i}<0$, we define $\varphi(i)$ for $i<\varphi(i)\le j $ that satisfies the following conditions.
  \begin{enumerate}
  \item $(-1)^{\varphi(i)-i}<0$ and $v(i,\varphi(i))<0$.\\
  \item If there exists $i<m<\varphi(i)$ such that $(-1)^{m-i}<0$, $v(i,m)<0$, then there must exist $i<k<m$ such that $(-1)^{m-k}<0$, $v(k,m)<0$.\\
  \item If there exists $i<n<\varphi(i)$ such that $(-1)^{\varphi(i)-n}<0$, $v(n,\varphi(i))<0$, then there must exist $n<k<\varphi(i)$ such that $(-1)^{k-n}<0$, $v(n,k)<0$.
    \end{enumerate}
    We further define $\psi(j)$ for $i\le \psi(j)< j $  that satisfies the following conditions.
    \begin{enumerate}
  \item $(-1)^{j-\psi(j)}<0$ and $v(\psi(j),j)<0$.\\
  \item If there exists $\psi(j)<m<j$ such that $(-1)^{m-\psi(j)}<0$, $v(\psi(j),m)<0$, then there must exist $\psi(j)<k<m$ such that $(-1)^{m-k}<0$, $v(k,m)<0$.\\
  \item If there exists $\psi(j)<n<j$ such that $(-1)^{j-n}<0$, $v(n,j)<0$, then there must exist $n<k<j$ such that $(-1)^{k-n}<0$, $v(n,k)<0$.
    \end{enumerate}

Then either $\varphi(i)$ or $\psi(j)$ must exist. In particular, if $\varphi(i)=j$, then $\psi(j)=i$.
\end{lemma}
\begin{proof}
For $\delta>1$, $i<j<i+\delta$, if $v(i,j):=(-1)^{c_{j}-c_i+j-i}<0$, then it is easy to verify that
\begin{align}\label{eq27}
v(i,j)=\prod\limits_{k=i}^{j-1}v(k,k+1).
\end{align}
We define the functions $\varphi$ and $\psi$ according to the lemma.

If $(-1)^{j-i}<0$, $v(i,j)<0$, then the lemma is trivial. Next we show, by contradiction, that either $\varphi(i)$ or $\psi(j)$ also exists for $(-1)^{j-i}>0$, $v(i,j)<0$.

Let us examine a simple case first. If $v(i,i+1)<0$ or $v(j-1,j)<0$ then $\varphi(i)=i+1$ or $\psi(j)=j-1$. Specifically, we can divide it into the following three cases.

\begin{align*}
\overset{i}{\mid}\ \ - \overset{\varphi(i)}{\mid}\cdots \ \cdots \ \cdots \ \overset{\psi(j)}{\mid}-\ \ \overset{j}{\mid},\ \overset{i}{\mid}\ \ - \overset{\varphi(i)}{\mid}\cdots \ \cdots \  \cdots \ \overset{j-1}{\mid}+\ \ \overset{j}{\mid},\ \overset{i}{\mid}\ \ + \overset{i+1}{\mid}\cdots \ \cdots \ \cdots \ \overset{\psi(j)}{\mid}-\ \ \overset{j}{\mid}.
\end{align*}
Here $\overset{m}{\mid}\ \cdots \ - \ \cdots  \ \overset{n}{\mid}$ means $v(m,n)<0$, and $\overset{m}{\mid}\ \cdots \ + \ \cdots  \ \overset{n}{\mid}$ means $v(m,n)>0$.

From this simple case, we can easily extend to the general case. Taking the positive sequence as an example, the reverse sequence is symmetrical with the positive sequence. If $v(i,i+1)>0$, $v(l,l+1)>0$, $(-1)^{l+1-i}>0$, and $v(k,k+1)<0$ for $1<k<l<j$, then $\varphi(i)=l+1$.
$$\overset{i}{\mid}\ \ + \overset{i+1}{\mid}- \overset{i+2}{\mid}-\overset{i+3}{\mid} \ \cdots \ - \ \cdots  \overset{l-2}{\mid} -\overset{l-1}{\mid} - \ \overset{l}{\mid}\ \ + \overset{l+1}{\mid}.$$
Similarly, if $v(j-1,j)>0$, $v(l-1,l)>0$, $(-1)^{l-j+1}<0$, and $v(k-1,k)<0$ for $1<l<k<j$, then $\psi(j)=l-1$.

Suppose we find such a counterexample column $\{c_k\}^j_{k=i}$, where $(-1)^{j-i}>0$, that has no corresponding $\varphi(i)$ or $\psi(j)$ to satisfy the definition in the lemma. According to the above discussion, either $v(i,i+1)>0$ and $v(j-1,j)>0$, or else $\psi(j)=j-1$ or $\varphi(i)=i+1$ must exist. Further, from our general case, our example must have $v(k,k+1)<0$ for $i+1 \le k\le i+2l$ and $j-2l'-1\le k \le j-2$, $l,l'\in\mathbb{N}$. Since $(-1)^{j-i}>0$, there must be an odd number of $k$ such that
\begin{align}\label{eq28}
v(k,k+1)>0 \ \  for \  \ i<k<j-1.
\end{align}
If there is only one $k$ value satisfying the above formula, that is, $i+2l+2=j-2l'-1$, $v(i+2l+1,i+2l'+2)>0$, then we consider $i_1=i$, $j_1=i+2l+2$, $i'_1=i+2l'+1$, $j'_1=j$. Then $(-1)^{j_1-i_1}(-1)^{j'_1-i'_1}=(-1)^{j-i+1}<0$, that is, $(-1)^{j_1-i_1}<0$ or $(-1)^{j'_1-i'_1}<0$. Hence, there must exist $\varphi(i)=j_1$ or $\psi(j)=i_1$, in contradiction to the supposition.
$$\overset{i_1}{\mid}  + \mid-\mid-\mid-\mid-\underset{i'_1}{\mid}+\overset{j_1}{\mid}-\mid-\mid-\mid + \underset{j'_1}{\mid}.$$

Next, we show three findings regarding the counterexample column $\{c_k\}^j_{k=i}$. First, we may find a sub-column of the counter-example column such that an element of this sub-column is still a counter-example to our lemma. Second, the number of elements in the sub-column that satisfies our supposition will become less and less according to our selection method. Third, we eventually find that the sub-column satisfying the assumption is always an example of the contradictions discussed above, proving the lemma by contradiction.
If there is more than one $k$ that satisfies Equation (\ref{eq28}), then we define
\begin{align*}
\underline{k}&:=\min\{k\mid v(k,k+1)>0, i<k<j-1\}, \\
\overline{k}&:=\max\{k\mid v(k,k+1)>0, i<k<j-1\}.
\end{align*}
For $i_1=i$, $j_1=\overline{k}+1$, we have $v(i_1,j_1)<0$, $(-1)^{j_1-i_1}<0$. By the definition of the functions $\varphi$, $\psi$ and the supposition, $\psi(j_1)>i_1$ must exist. Similarly, we consider $i'_1=\underline{k}$, $j'_1=j$, in which case $\varphi(i'_1)<j_1$ must exist.
$$\overset{i_1}{\mid}  + \mid-\mid-\mid\ \cdots \underset{i'_1}{\mid}+\ \mid\ \cdots \ -\ \cdots \ \mid\ +\overset{j_1}\mid\ \cdots \ \mid-\mid-\mid + \underset{j'_1}{\mid}.$$
Moreover, since $(-1)^{j_1-\psi(j_1)}<0$, we have $\psi(j_1)>i'_1$. Similarly, we have $\varphi(i'_1)<j_1$. We further note that $\psi(j_1) \neq \varphi(i'_1)$. In fact, if $\psi(j_1) = \varphi(i'_1)$, then by Equation (\ref{eq27}), we have $$v(i_1,j'_1)=v(i_1,i'_1)v(i'_1,\varphi(i'_1))v(\psi(j_1),j_1)v(j_1,j'_1)>0,$$ which contradicts the supposition. We now describe three cases as follows.

In the first case, $\psi(j_1)<\varphi(i'_1)$. Then, by Equation (\ref{eq27}), the sign of $v(m,n)$ has only the following unique case for $i'_1\le m,n\le j_1$.
$$\underset{i'_1}{\mid}\ \cdots \ +\ \cdots \overset{\psi(j_1)}{\mid} \ \cdots  \ - \ \cdots   \   \underset{\varphi(i'_1)}{\mid}\cdots \ +\ \cdots \  \overset{j_1}{\mid}.$$
Then we let $i_2=\psi(j_1)$ and $j_2=\varphi(i'_1)$, and we consider the sub-column $\{c_k\}^{j_2}_{k=i_2} \subset \{c_k\}^{j}_{k=i}$. The number of elements in column $\{c_k\}^{j_2}_{k=i_2}$ at this stage is less than the number of elements in column $\{c_k\}^{j}_{k=i}$, and if there is only one $k_2$ such that $v(k_2,k_2+1)>0$ for $i_2<k_2<j'_2-1$, then is the supposition is contradicted. So, we retreat to the case in which there are multiple $k_2$ values in the sub-column satisfying the above formula. In this case, there will always be contradictions in the supposition as the number of elements decreases. At this stage, it is easy to verify that it satisfies the above case.

If $\psi(j_1)>\varphi(i'_1)$.  $v(i'_1,i'_1+1)>0$ and $v(j_1-1,j_1)>0$, so $\varphi(i'_1)>i'_1+1$, and $\psi(j_1)<j_1-1$. When we consider $i_2=\varphi(i'_1)$ and $j_2=\psi(j_1)$, then $$(-1)^{j_2-i_2}=(-1)^{j_1-i'_1}(-1)^{j_2-j_1}(-1)^{i'_1-i_2}>0,$$
 and there must be an odd number of $k_2$ values such that $v(k_2,k_2+1)>0$ for $i_2<k_2<j_2-1$.
$$\overset{i_1}{\mid}  + \mid-\mid-\mid\ \cdots \underset{i'_1}{\mid}+ \mid\ \cdots \ -\ \cdots \ \mid+\underset{i_2}{\mid} \ \cdots  \ - \ \cdots   \   \overset{j_2}{\mid}+\mid\ \cdots \ -\ \cdots \ \  \mid + \overset{j_1}{\mid}\ \cdots \ \mid-\mid-\mid + \underset{j'_1}{\mid}.$$
Similarly, we define
\begin{align*}
&\underline{k_2}:=\min\{k\mid v(k,k+1)>0, i_2<k<j_2-1\},\\
&\overline{k_2}:=\max\{k\mid v(k,k+1)>0, i_2<k<j_2-1\}.
\end{align*}
If $\underline{k_2}=\overline{k_2}$, then the definition of $\psi(j_1)$ or $\varphi(i'_1)$ is violated. So $\underline{k_2}<\overline{k_2}$, and we consider the following two cases.

In the second case, $(-1)^{\underline{k_2}-i_2}>0$ or $(-1)^{j_2-1-\overline{k_2}}>0$. Here, without losing generality, we can suppose that $(-1)^{\underline{k_2}-i_2}>0$.
We consider $i_3=i_1$, $j_3=\underline{k_2}+1$, and by the supposition there must exist $\psi(j_3)>i'_1$. Now we consider $i_4=\psi(j_3)$, $j_4=i_2-1$, $i'_4=\psi(j_3)+1$, $j'_4=i_2$. Here the sub-column $\{c_k\}^{j'_4}_{k=i_4} \subset \{c_k\}^{j}_{k=i}$.
$$\overset{i_3}{\mid}  + \mid-\mid-\mid\ \cdots \underset{i'_1}{\mid}\ \cdots \ + \ \cdots \overset{\psi(j_3)}{\mid} + \ \ \underset{i'_4}{\mid}\cdots  \ - \ \cdots   \overset{j_4}{\mid}+ \underset{j'_4}{\mid}+\ \mid\ \cdots \ \mid-\mid-\mid + \overset{j_3}{\mid}.$$
It is easy to verify that the sub-column $\{c_k\}^{j'_4}_{k=i_4}$ satisfies the initial case. Similarly, we note that the number of elements that satisfy the initial situation is reduced. As mentioned above, this reduction will always lead to a contradiction of the supposition.

In the third case, $(-1)^{\underline{k_2}-i_2}<0$, and $(-1)^{j_2-1-\overline{k_2}}<0$. Here we consider $i_3=i_1$, $j_3=\overline{k_2}+1$, $i'_3=\underline{k_2}$, and $j'_3=j'_1$. Then $\varphi(i'_3)>i'_3$ and $\psi(j_3)<j_3$ must exist.
\begin{align*}
&\overset{i_3}{\mid}  + \mid-\mid-\mid\ \cdots \underset{i'_1}{\mid}\ \ \cdots \ -\ \cdots  \underset{\varphi(i'_1)}{\mid}-\ \ \mid\cdots \underset{i'_3}{\mid}+\mid\ \cdots \ -\ \cdots \   \mid\ +\overset{j_3}{\mid}\cdots \\
&\cdots \mid\ \ -\overset{\psi(j_1)}{\mid}\cdots \ -\ \cdots \ \  \overset{j_1}{\mid}\cdots \ \mid-\mid-\mid + \underset{j'_3}{\mid}.
\end{align*}
If $\psi(j_3)<\varphi(i'_3)$, then the logic proceeds as in the first case. If $\psi(j_3)>\varphi(i'_3)$, we can still find the second and third cases corresponding to it, but as we further analyze the middle part $i'_3\le k\le j_3$, the number of occurrences of the third case will decrease because the number of middle elements decreases with the number of choices. Thus, the third case will always transform into the second case or else yield a contradiction. Hence, we complete the proof.
\end{proof}
\begin{remark}
The existence of $\varphi$ and $\psi$ is related to the length of the product term, but the value of the function, if present, is independent of the product term's length. We default to considering product terms that include any negative points. It can be seen from the above lemma that any starting point $c_i$ has only one unique end point corresponding to it, so we only need the parameter $i$, representing the end point, to formulate $\varphi(i)$. Similarly, there is only one unique starting point corresponding to any end point $j$, so we only need the parameter $j$, representing to represent the starting point, to formulate $\psi(j)$.

According to the factorization method of the above lemma, if there is a corresponding relationship between each negative point, it will be a one-to-one mapping relationship. In other words, $\psi$ is the inverse function of $\varphi$. Further, we have the following corollary, which is an important result for our later discussion.
\end{remark}

\begin{lemma}\label{pr2}
  Consider $m$ and $n$, assuming the existence of $\varphi(m)$ and $\varphi(n)$. If $m\le n\le \varphi(m) < \varphi(n)$, then $v(m,n)<0$ and $v(\varphi(m),\varphi(n))<0$.
\end{lemma}
\begin{proof}
 By the definition of $\varphi$, we know that $v(n,\varphi(n))<0$ and $v(m,\varphi(n))<0$. If $m=n$, then according to Lemma \ref{le6}, the definition of $\varphi(n)$ is violated.
 Now, we instead consider $m<n$, and we follow the notation of Lemma \ref{le6}. The sign of the function $v$ between the negative point $c_m$ and the negative point $c_{\varphi(n)}$ has the following two cases.
 $$\underset{m}{\mid}\cdots \ +\ \cdots \overset{n}{\mid} \ \ \cdots  \ - \ \cdots    \underset{\varphi(m)}{\mid}\cdots \ +\ \cdots \overset{\varphi(n)}{\mid},$$
$$\underset{m}{\mid} \cdots \ -\ \cdots \overset{n}{\mid} \ \ \cdots  \ + \ \cdots      \underset{\varphi(m)}{\mid}\cdots \ -\ \cdots  \overset{\varphi(n)}{\mid}.$$
For the first case, we have $v(n, \varphi(m)) < 0$, and by Lemma \ref{le6}, then the definition of $\varphi(n)$ or $\varphi(m)$ is violated. This leaves only the second case, in which $v(m,n)<0$ and $v(\varphi(m),\varphi(n))<0$.

\end{proof}

In our example, if we fix the starting point $i$ of the product term and iterate according to the iteration direction of the first method, then when the end point of the product term is negative point $c_j$, there may be $c_{\psi(j)}>i$. At this stage we can regard $-1<\nu(\psi(j),j)<0$ as a new element and continue to iterate, specifically
$$\mu(i,c_{\psi(j)}-1)+\mu(i,c_j)=\mu(i,c_{\psi(j)}-1)(1-\mid\nu(\psi(j),j)\mid).$$
In this way, we can skip some negative points and continue to iterate, but once there is no $\psi(j)$ or $c_{\psi(j)}\le i$, then the iteration stops. Then, we fix any starting point, which can be the sum of several polynomials according to the additive form of Equation (\ref{eq26}). In other words, the sum of the product terms starting at any starting point can only stop at some negative point if the iteration stops early. Now, let us take into account the iterative approach in the form of Lemma \ref{le5}. Due to the characteristics of one-to-one mapping, the product terms involved in the iterative routes in the same direction as that of the fixed starting point product terms differ for each route. For example, if  $\boldsymbol{a}=(1,2,1,2,2,1,1,2,2,2)$, then we have the following.
$$
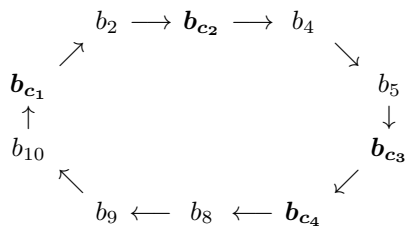
\begin{matrix}
  \  & \  & b_2 & \longrightarrow & \boldsymbol{b_{c_2}}  & \longrightarrow & b_4 & \  & \  \\
  \  & \nearrow & \  & \  & \  & \ & \ & \searrow & \  \\
  \boldsymbol{b_{c_1}} & \  & \  & \  & \  & \  & \ & \ & b_5 \\
  \uparrow & \  & \  & \  & \  & \  & \ & \ & \downarrow\\
   b_{10} & \  & \  & \  & \  & \  & \ & \ & \boldsymbol{b_{c_3}} \\
  \  & \nwarrow & \  & \  & \  & \ & \ & \swarrow & \  \\
  \  & \  & b_9 & \longleftarrow & b_8 &\longleftarrow & \boldsymbol{b_{c_4}}  & \  & \
\end{matrix}
$$\\
If we take $1$ as the starting point and follow the first type of iteration direction, then the sum of products can be divided into four polynomials corresponding to the following four routes.
$$1\leftarrow b_{c_1}b_2b_{c_2}\leftarrow b_4 \leftarrow b_5,$$
$$-b_{c_1}\leftarrow b_2,$$
$$b_{c_1}b_2b_{c_2}b_4b_5b_{c_3},$$
$$-b_{c_1}b_2b_{c_2}b_4b_5b_{c_3}b_{c_4}\leftarrow b_8\leftarrow b_9.$$
Then, according to our iteration rule the route is unique and the end points of the product terms are not repeated. And the sign of each polynomial is only related to the product term involved in the iteration end point. In our example the first polynomial is $1$ which is positive, the second polynomial is $-b_{c_1}$ which is positive, the third is $b_{c_1}b_2b_{c_2}b_4b_5b_{c_3}$ which is negative and the fourth is $-b_{c_1}b_2b_{c_2}b_4b_5b_{c_3}b_{c_4}$ which is negative.

 We note that the iteration relationship between positive points is adjacent, while the iteration relationship between negative points consists of an abstract span. However, according to the lemma \ref{le6}, if the order of the negative points in the cycle is fixed, then the route generated by the part related to the negative points according to the first iteration direction is also fixed. In this way, we can study the relationship between the sum of the product terms starting from any negative point and the sum of the product terms starting from the positive point between this negative point and the previous negative point.

 For the general case, we first introduce some functions to facilitate the description of the construction of polynomials.

  \begin{definition}\label{de4}
  For any $i$, $j$, $0<j-i<\delta$, if $\varphi(i)$ (resp. $\psi(j)$) does not exist, we extend the definition $\varphi(i)=i-1$ (resp. $\psi(j)=j+1$). Define
      \begin{align*}
\psi_1(i,j)&=\psi(j)\ \ \psi_k(i,j)=\psi(\psi_{k-1}(i,j)-1),\\
\tilde{\psi}(i,j)&=\min\{\psi_k(i,j)\mid\psi_k(i,j)\ge i\},\\
 \tilde{\psi}^*(i,j)&=
 \begin{cases}
\tilde{\psi}(i,j)& \text{if $\tilde{\psi}(i,j)>i$} \\
\varphi(i)+1 & \text{if $\tilde{\psi}(i,j)=i$  }.
\end{cases}
          \end{align*}
We abbreviate $\psi_k(i,j)$ (respectively $\tilde{\psi}(i,j)$, $\tilde{\psi}^*(i,j)$) to $\psi_k(j)$ (resp. $\tilde{\psi}(j)$, $\tilde{\psi}^*(j)$).
  \end{definition}

Let us take the negative point $c_{i}$ as the starting point of the product term and take $c_i$ to be odd. At this stage, we consider the sum of
the product terms of lengths $0$ to $2h-1$. To simplify the explanation, we will not consider the case of $c_i-c_{i-1}=1$, in which no negative point exists between $c_{i-1}$ and $c_i$. Then, if $j+1<i+\delta$, the product term polynomial corresponds to the following general route.

 \begin{align*}
 &\nu(i,\tilde{\psi}(j)-1)\leftarrow  b_{c_{\tilde{\psi}(j)-1}+1}\cdots \leftarrow b_{c_{\tilde{\psi}(j)}-1}\leftarrow\cdots \leftarrow b_{c_{\psi_2(j)-1}+1}\leftarrow \cdots \leftarrow b_{c_{\psi_2(j)}-1} \\
 &\leftarrow b_{c_{\psi_2(j)}}\cdots b_{c_{\psi_1(j)-1}}\leftarrow b_{c_{\psi_1(j)-1}+1}\leftarrow \cdots \leftarrow b_{c_{\psi_1(j)}-1}\leftarrow b_{c_{\psi_1(j)}}\cdots b_{c_j}\leftarrow b_{c_j+1}\leftarrow\\
 &\cdots \leftarrow b_{c_{j+1}-1}\\
\end{align*}
In particular, if $\psi_1(j)=j+1$, the iterative route is a very simple form as follows.
$$\nu(i,j)\leftarrow b_{c_j+1}\leftarrow\cdots \leftarrow b_{c_{j+1}-1}.$$
Here $\varphi(j+1)\ge i+\delta$ or else $\varphi(j+1)$ does not exist. Note that there is not necessarily a positive point between every two negative points, so if the serial number of the point on the left side of the arrow is greater than the serial number of the point on the right side of the arrow, then this part of the route does not exist. For the time being, we don't consider the situation that there is no positive point between the two negative points. Furthermore, if the midway route $\psi$ does not exist, then it may be the case that $\psi_{k+1}(j)=\psi_k(j)$. Then, this part of the iterative route does not exist, and the iteration stops. Due to the restriction on the length of the product term, we have $j+1<i+\delta$, otherwise we iterate from end point $b_{c_{j+1}-2}$.\\
If $c_i$ is even, for the above route, we only need to change $\tilde{\psi}(j)$ to $\tilde{\psi}^*(j)$. In fact, due to the limitation of the length of the product term, no product term exists with a length of $0$ at this time, that is, $1$. Then, for a specific $j$ such that $\tilde{\psi}(j) =i$, the iterative route ends at $\varphi(i)$.

We note that if there are positive points between $c_{i-1}$ and $c_i$, the iterative route for the polynomial with these positive points as the origin of the product term is the same with respect to the negative point $c_i$  as the iterative route for the product term. Except that, we need to consider the part from the positive point to $c_i$, and the part from $c_{i+\delta-1}$ to the before the positive point.

Going back to our example, we study the polynomial of the product term starting at the positive point between $c_0$ and $c_1$. We take $10$, $9$, and $8$ as starting points and divide the sum of the products of each starting point into four polynomials corresponding to the following twelve routes.
\begin{align*}
&\qquad{}-b_{10} \leftarrow b_{c_1}b_2b_{c_2}\leftarrow b_4 \leftarrow b_5, &\quad &1\leftarrow b_{9} \leftarrow b_{10} \leftarrow b_{c_1}b_2b_{c_2}\leftarrow b_4 \leftarrow b_5, \\
&b_{10}b_{c_1} \leftarrow b_2, &\quad  &\qquad{}-b_{9}b_{10}b_{c_1} \leftarrow b_2, \\
&\qquad{}-b_{10}b_{c_1}b_2b_{c_2}b_4b_5b_{c_3},&\quad &b_9b_{10}b_{c_1}b_2b_{c_2}b_4b_5b_{c_3},\\
&b_{10}b_{c_1}b_2b_{c_2}b_4b_5b_{c_3}b_{c_4}\leftarrow b_8 \leftarrow b_9.&\quad &\qquad{}-b_9b_{10}b_{c_1}b_2b_{c_2}b_4b_5b_{c_3}b_{c_4}. \\
\end{align*}
\begin{align*}
&\qquad{}-b_{8}\leftarrow b_{9}\leftarrow b_{10} \leftarrow b_{c_1}b_2b_{c_2}\leftarrow b_4 \leftarrow b_5,  \\
&b_{8}b_{9}b_{10}b_{c_1} \leftarrow b_2, \\
&\qquad{}-b_{8}b_{9}b_{10}b_{c_1}b_2b_{c_2}b_4b_5b_{c_3},\\
&b_{8}b_{9}b_{10}b_{c_1}b_2b_{c_2}b_4b_5b_{c_3}b_{c_4}.\\
\end{align*}
We note that the second and third routes corresponding to each starting point are similar to the route of the product term with $c_1$ as the starting point, but we choose only polynomials that are positive in $c_1$ for the second kind of iteration as follows.
$$b_{8}\rightarrow b_{9}\rightarrow b_{10}\rightarrow -b_{c_1} \leftarrow b_2,$$
$$b_{8}\rightarrow b_{9}\rightarrow -b_{10}b_{c_1}b_2b_{c_2}b_4b_5b_{c_3}.$$
For the first route of each starting point, we first follow the second type of iterative method, and then for the redundant part, we can always compose it into two positive polynomials.
\begin{align*}
b_8 \rightarrow b_9 \rightarrow b_{10} \rightarrow &1\leftarrow b_{c_1}b_2b_{c_2}\leftarrow b_4 \leftarrow b_5,\\
b_8 \rightarrow &1\leftarrow  b_{9}.
\end{align*}
For the fourth route of each starting point, the polynomial of the fourth route is negative in $c_1$, so we can still compose two positive polynomials from three positive starting points.
$$b_9 \rightarrow b_{10}b_{c_1}b_2b_{c_2}b_4b_5b_{c_3}b_{c_4}\leftarrow b_8,$$
$$b_{8}b_{9}b_{10}b_{c_1}b_2b_{c_2}b_4b_5b_{c_3}b_{c_4}.$$
In this way, we keep the route of the negative polynomial starting from $c_1$, the route of the positive polynomial starting from $8910c_1$, and other parts strictly positive.
$$b_{c_1}b_2b_{c_2}b_4b_5b_{c_3},$$
$$-b_{c_1}b_2b_{c_2}b_4b_5b_{c_3}b_{c_4}\leftarrow b_8\leftarrow b_9,$$
$$b_{8}b_{9}b_{10}b_{c_1} \leftarrow b_2,$$
$$-b_{8}b_{9}b_{10}b_{c_1}b_2b_{c_2}b_4b_5b_{c_3}.$$

For the general case, if $c_i$ is odd, $\nu(i,\tilde{\psi}(j)-1)>0$ and $j+1<i+\delta$, then we can integrate the polynomial routes starting from the positive points between the two negative points into a positive polynomial route.
\begin{equation}\label{eq29}
\begin{aligned}
&b_{c_{i-1}+1}\rightarrow\cdots \rightarrow b_{c_i-1}\rightarrow \nu(i,\tilde{\psi}(j)-1)\leftarrow b_{c_{\tilde{\psi}(j)-1}+1}\\
&\leftarrow\cdots \leftarrow b_{c_{\tilde{\psi}(j)}-1} \leftarrow\cdots \leftarrow b_{c_{j+1}-1}.
\end{aligned}
\end{equation}
Furthermore, if $\nu(i,\tilde{\psi}(j)-1)<0$, $j+1<i+\delta$, we can still integrate into a positive polynomial.
\begin{align}\label{eq30}
&b_{c_{i-1}+1}\rightarrow\cdots \rightarrow b_{c_i-2}\rightarrow -b_{c_i-1}\nu(i,\tilde{\psi}(j)-1)\leftarrow b_{c_{\tilde{\psi}(j)-1}+1}\leftarrow\cdots \leftarrow b_{c_{\tilde{\psi}(j)}-1}\nonumber\\
& \leftarrow\cdots \leftarrow b_{c_{j+1}-1}.
\end{align}
For the polynomial routes starting from a positive point is the starting point, and the ending points are before negative point $c_i$, with $c_i-c_{i-1}$ even, we can integrate them into the following positive polynomial routes.
\begin{equation}\label{eq31}
\begin{aligned}
b_{c_i-3}\rightarrow &1 \leftarrow b_{c_i-2},\\
b_{c_i-5}\rightarrow &1 \leftarrow b_{c_i-4} \leftarrow b_{c_i-3} \leftarrow b_{c_i-2},\\
&\cdots ,\\
b_{c_i-2l-1}\rightarrow &1\leftarrow b_{c_i-2l} \leftarrow \cdots  \leftarrow b_{c_i-2},\\
&\cdots ,\\
b_{c_{i-1}+1}\rightarrow &1\leftarrow b_{c_{i-1}+2} \leftarrow \cdots  \leftarrow b_{c_i-2}.\\
\end{aligned}
\end{equation}
If $c_i-c_{i-1}$ is odd, we keep the last route and connect it with the route of $\nu(i,\tilde{\psi}(j)-1)=1$ to form a positive polynomial.
\begin{align}\label{eq32}
1\leftarrow b_{c_{i-1}+1} \leftarrow \cdots  \leftarrow b_{c_i-1} \leftarrow b_{c_i}\cdots b_{c_{\varphi(i)}} \leftarrow\cdots \leftarrow b_{c_{j+1}-1}.
\end{align}
If $\nu(i,\tilde{\psi}(j)-1)>0$, $j+1=i+\delta$, $\nu(i,j)<0$, we ignore for now the part in which the end point of the product term is between $c_{i+\delta-1}$ and $c_{i+\delta}$.
\begin{equation}\label{eq33}
\begin{aligned}
&b_{c_{i-1}+1}\rightarrow\cdots \rightarrow b_{c_i-1}\rightarrow \nu(i,\tilde{\psi}(j)-1)\leftarrow b_{c_{\tilde{\psi}(j)-1}+1}\leftarrow\cdots \\
&\leftarrow b_{c_{\tilde{\psi}(j)}-1} \leftarrow\cdots \leftarrow b_{c_{\psi_1(j)}}\cdots b_{c_j}.
\end{aligned}
\end{equation}
When the starting point is a positive point, the end points are the product items between negative point $c_{i+\delta-1}$ and negative point $c_{i+\delta}$. If $c_i-c_{i-1}$ is even, we can integrate them into the following positive polynomial routes.
\begin{equation}\label{eq34}
\begin{aligned}
b_{c_i-1} \rightarrow &\mu(c_{i},c_{i+\delta-1}+1)\leftarrow b_{c_{i+\delta-1}+2}\leftarrow \cdots  \leftarrow b_{c_{i+\delta}-2},\\
b_{c_i-3} \rightarrow &\mu(c_{i}-2,c_{i+\delta-1}+1) \leftarrow b_{c_{i+\delta-1}+2}\leftarrow \cdots  \leftarrow b_{c_{i+\delta}-4},\\
&\cdots ,\\
b_{c_i-2l-1}\rightarrow &\mu(c_{i}-2l,c_{i+\delta-1}+1)\leftarrow  b_{c_{i+\delta-1}+2}\leftarrow \cdots  \leftarrow b_{c_{i+\delta}-2l-2},\\
&\cdots ,\\
b_{c_{i-1}+3}\rightarrow& \mu(c_{i-1}+4,c_{i+\delta-1}+1)\leftarrow  b_{c_{i+\delta-1}+2}.\\
\end{aligned}
\end{equation}
On the other hand, if $c_i-c_{i-1}$ is odd, the last route is of the following form.
\begin{align}\label{eq35}
\mu(c_{i-1}+3,c_{i+\delta-1}+1)\leftarrow b_{c_{i+\delta-1}+2}.
\end{align}
If $\nu(i,\tilde{\psi}(j)-1)>0$, $j+1=i+\delta$, $\nu(i,j)>0$, we ignore for now the part in which the end point of the product term is between $c_{\psi(i+\delta-1)}$ and $c_{i+\delta}$.
\begin{equation}\label{eq36}
\begin{aligned}
&b_{c_{i-1}+1}\rightarrow\cdots \rightarrow b_{c_i-1}\rightarrow \nu(i,\tilde{\psi}(j)-1)\leftarrow b_{c_{\tilde{\psi}(j)-1}+1}\\
&\leftarrow\cdots \leftarrow b_{c_{\tilde{\psi}(j)}-1} \leftarrow\cdots \leftarrow b_{c_{\psi_1(j)}-1}
\end{aligned}
\end{equation}
Similarly, if $c_i-c_{i-1}$ is even, we can integrate these remaining terms into the following positive polynomials.
\begin{equation}\label{eq37}
\begin{aligned}
b_{c_i-1} \rightarrow &\nu(i,i+\delta-1)\leftarrow b_{c_{i+\delta-1}+1}\leftarrow \cdots  \leftarrow b_{c_{i+\delta}-2},\\
b_{c_i-3} \rightarrow &\mu(c_{i}-2,c_{i+\delta-1}) \leftarrow b_{c_{i+\delta-1}+1}\leftarrow \cdots  \leftarrow b_{c_{i+\delta}-4},\\
&\cdots ,\\
b_{c_i-2l-1}\rightarrow &\mu(c_{i}-2l,c_{i+\delta-1}) \leftarrow  b_{c_{i+\delta-1}+1}\leftarrow \cdots  \leftarrow b_{c_{i+\delta}-2l-2},\\
&\cdots ,\\
b_{c_{i-1}+1}\rightarrow&\mu(c_{i-1}+2,c_{i+\delta-1}).\\
\end{aligned}
\end{equation}
Finally, if $c_i-c_{i-1}$ is odd, then the form of the last route is as follows.
\begin{align}\label{eq38}
b_{c_{i-1}+2}\rightarrow\mu(c_{i-1}+3,c_{i+\delta-1})\leftarrow b_{c_{i+\delta-1}+1}.
\end{align}
Similarly, if $\nu(i,\tilde{\psi}(j)-1)<0$, we keep the negative polynomial starting from the negative point, and the remaining polynomials starting from the positive point can also be integrated into the sum of several positive polynomials through dislocation. The same is true for the case in which $c_i$ is an even number. Compared with the case of odd $c_i$, the only difference is in the length of the product term. It is always possible to integrate polynomial routes with different starting points and a common end point, with the end point not falling within the area involved by the starting point, into a route similar to (\ref{eq29}) (\ref{eq30}) (\ref{eq33}) (\ref{eq36}). Further, polynomial routes with different starting points and a common end point, with the end point present in the area involved by the starting point, can always be integrated into routes similar to (\ref{eq31}) (\ref{eq34}) (\ref{eq35}) (\ref{eq37}) (\ref{eq38}). However, unlike the case where $c_i$ is an odd number, there is currently no route similar to (\ref{eq32}).

Due to the arbitrariness of $c_i$, any negative polynomial of the product term starting from a positive point can always find a unique positive polynomial corresponding to it so that the addition of the two is strictly positive. In this way, our study of $Q$ is reduced to the study of the product term of the starting point at the negative point and at the point next to the negative point, regardless of the existence of positive points between the negative points. To summarize our discussion results, the function $Q$ can be preliminarily calculated in the form of the following lemma.

\begin{lemma}\label{le7}
The structure function $Q$ of non-random strategy $D$ is greater than the following polynomial.
\begin{align*}
Q&>\sum\limits_{i=1}^{\delta}\sum\limits_{j=i}^{i+\delta-1}\mathbb{I}_{A_i}\mathbb{I}_{C^j_i}\mathbb{I}\{\nu(i,\tilde{\psi}(j)-1)<0\}\nu(i,\tilde{\psi}(j)-1)\eta(\tilde{\psi}(j)-1,j+1)\\
&\qquad{}+\sum\limits_{i=1}^{\delta}\sum\limits_{j=i}^{i+\delta-2}\mathbb{I}_{A_{i-1}}\mathbb{I}_{C^j_{i-1}}\mathbb{I}\{\nu(i-1,\tilde{\psi}^*(j)-1)>0\}\tilde{\nu}(i-1,\tilde{\psi}^*(j)-1)\eta(\tilde{\psi}^*(j)-1,j+1)\\
&\qquad{}+\sum\limits_{i=1}^{\delta}\sum\limits_{j=i}^{i+\delta-1}\mathbb{I}_{B_i}\mathbb{I}_{C^j_i}\mathbb{I}\{\nu(i,\tilde{\psi}^*(j)-1)<0\}\nu(i,\tilde{\psi}^*(j)-1)\eta(\tilde{\psi}^*(j)-1,j+1)\\
&\qquad{}+\sum\limits_{i=1}^{\delta}\sum\limits_{j=i}^{i+\delta-2}\mathbb{I}_{B_{i-1}}\mathbb{I}_{C^j_{i-1}}\mathbb{I}\{\nu(i-1,\tilde{\psi}(j)-1)>0\}\tilde{\nu}(i-1,\tilde{\psi}(j)-1)\eta(\tilde{\psi}(j)-1,j+1).\\
\end{align*}
Here $$A_i:=\{c_i=2l+1, l\in \mathbb{N}\}, \ \ B_i:=\{c_i=2l, l\in \mathbb{N}\},$$
$$C^j_i:=\{\varphi(j+1)\ge i+\delta \ \ or\ \ \varphi(j+1)=j\}.$$
$\eta(\tilde{\psi}(j)-1,j+1)$ is a polynomial generated by the iterative route below.
\begin{align*}
&1\leftarrow  b_{c_{\tilde{\psi}(j)-1}+1}\cdots \leftarrow b_{c_{\tilde{\psi}(j)}-1}\leftarrow \cdots \leftarrow b_{c_{\psi_2(j)}}\cdots b_{c_{\psi_1(j)-1}}\leftarrow \cdots \leftarrow b_{c_{\psi_1(j)}}\cdots b_{c_j}\nonumber\\
&\leftarrow b_{c_j+1}\leftarrow\cdots \leftarrow b_{c_{j+1}-1}.
\end{align*}
\end{lemma}
\begin{proof}
Using Equation (\ref{eq26}) and Equation (\ref{eq28}), we summarize the results of the discussion above. We reorganize the product term polynomials starting at the positive point into several positive polynomials. In the remaining part, we retain the negative polynomial starting from the negative point $c_i$, as well as some positive polynomials starting from $c_i+1$.

We fix $c_i=2l+1$, $l\in \mathbb{N}$, $i\le j \le i+\delta-1$, and then we consider a product term of length ranging from $0$ to $2h-1$ starting from negative point $c_i$. For convenience, we also temporarily consider a product term of length of $2h$. Then, the iterative route has the general form.
\begin{align*}
 &\nu(i,\tilde{\psi}(j)-1)\leftarrow  b_{c_{\tilde{\psi}(j)-1}+1}\cdots \leftarrow b_{c_{\tilde{\psi}(j)}-1}\leftarrow\cdots \leftarrow b_{c_{\psi_2(j)-1}+1}\leftarrow \cdots \leftarrow b_{c_{\psi_2(j)}-1} \\
 &\leftarrow b_{c_{\psi_2(j)}}\cdots b_{c_{\psi_1(j)-1}}\leftarrow b_{c_{\psi_1(j)-1}+1}\leftarrow \cdots \leftarrow b_{c_{\psi_1(j)}-1}\\
 &\leftarrow b_{c_{\psi_1(j)}}\cdots b_{c_j}\leftarrow b_{c_j+1}\leftarrow\cdots \leftarrow b_{c_{j+1}-1},\\
\end{align*}
where $\varphi(j+1)\ge i+\delta$ or $\varphi(j+1)=j$, $\nu(i,\tilde{\psi}(j)-1)<0$. Then, the sum of all routes in this part can be written as the following polynomial.
\begin{align*}
&\sum\limits_{i=1}^{\delta}\sum\limits_{j=i}^{i+\delta-1}\mathbb{I}_{A_i}\mathbb{I}_{C^j_i}\mathbb{I}\{\nu(i,\tilde{\psi}(j)-1)<0\}\nu(i,\tilde{\psi}(j)-1)\eta(\tilde{\psi}(j)-1,j+1)\\
&\qquad{}-\sum\limits_{i=1}^{\delta}\mathbb{I}_{A_i}\mathbb{I}\{\nu(i,\tilde{\psi}(i+\delta-1)-1)<0\}\mu(1,2h)
\end{align*}
Here $\eta(\tilde{\psi}(j)-1,j+1)$ is a polynomial generated by the iterative route above, and $$A_i:=\{c_i=2l+1, l\in \mathbb{N}\} \ \ C^j_i:=\{\varphi(j+1)\ge i+\delta \ \ or\ \ \varphi(j+1)=j\}.$$

As for the case in which the product term starts at $c_{i-1}+1$ and $c_{i-1}$ is odd, we note that $c_{i-1}+1+2h-1=c_{i+\delta-1}$, and we discuss the case of $j=i+\delta-1$ in the summation process separately. Then, the iterative route has the following general form.
$$\tilde{\nu}(i-1,\tilde{\psi}^*(j)-1)\leftarrow b_{c_{\tilde{\psi}^*(j)-1}+1}\leftarrow\cdots \leftarrow b_{c_{\tilde{\psi}^*(j)}-1} \leftarrow\cdots \leftarrow b_{c_{j+1}-1}.$$
Here $\varphi(j+1)\ge i+\delta-1$ or $\varphi(j+1)=j$, and $\tilde{\nu}(i-1,\tilde{\psi}^*(j)-1)<0$.
Then, the sum of all routes in this part can be written as the following polynomial.
\begin{align*}
&\sum\limits_{i=1}^{\delta}\sum\limits_{j=i}^{i+\delta-2}\mathbb{I}_{A_{i-1}}\mathbb{I}_{C^j_{i-1}}\mathbb{I}\{\nu(i-1,\tilde{\psi}^*(j)-1)>0\}\tilde{\nu}(i-1,\tilde{\psi}^*(j)-1)\eta(\tilde{\psi}^*(j)-1,j+1)\\
&\qquad{}+\sum\limits_{i=1}^{\delta}\mathbb{I}_{A_{i-1}}\mathbb{I}\{\nu(i-1,\tilde{\psi}^*(i+\delta-1)-1)>0\}\mu(1,2h).
\end{align*}
Similarly, for $c_{i-1}$, with $c_i$ even, and $i\le j \le i+\delta-1$, we have
\begin{align*}
&\sum\limits_{i=1}^{\delta}\sum\limits_{j=i}^{i+\delta-1}\mathbb{I}_{B_i}\mathbb{I}_{C^j_i}\mathbb{I}\{\nu(i,\tilde{\psi}^*(j)-1)<0\}\nu(i,\tilde{\psi}^*(j)-1)\eta(\tilde{\psi}^*(j)-1,j+1)\\
&\qquad{}+\sum\limits_{i=1}^{\delta}\sum\limits_{j=i}^{i+\delta-2}\mathbb{I}_{B_{i-1}}\mathbb{I}_{C^j_{i-1}}\mathbb{I}\{\nu(i-1,\tilde{\psi}(j)-1)>0\}\tilde{\nu}(i-1,\tilde{\psi}(j)-1)\eta(\tilde{\psi}(j)-1,j+1),
\end{align*}
where $B_i:=\{c_i=2l, l\in \mathbb{N}\}$. Next, we show that
$$\sum\limits_{i=1}^{\delta}\mu(1,2h)\bigg(\mathbb{I}_{A_{i-1}}\mathbb{I}\{\nu(i-1,\tilde{\psi}^*(i+\delta-1)-1)>0\}-\mathbb{I}_{A_i}\mathbb{I}\{\nu(i,\tilde{\psi}(i+\delta-1)-1)<0\}\bigg)\ge0.$$
We fix $1\le i\le \delta$. Then if $c_i$ is odd, with $\nu(i,\tilde{\psi}(i+\delta-1)-1)<0$, then $\tilde{\psi}(i+\delta-1)-1>i$. By the definition of $\tilde{\psi}$, we have $\tilde{\psi}(i+\delta-1)=\tilde{\psi}^*(i+\delta-1)$. According to Lemma \ref{le6}, we know that there must exist $i<\varphi(i)<\tilde{\psi}^*(i+\delta-1)-1$ with $c_{\varphi(i)}$ odd such that \\
$\nu(\varphi(i),\tilde{\psi}^*(i+\delta-1)-1)>0$. Now let $i'=\varphi(i)+1$. By the definition of $\tilde{\psi}^*$, we know that $$\tilde{\psi}^*(i',i'+\delta-1)=\tilde{\psi}^*(i',\varphi(i+\delta))=\tilde{\psi}^*(\varphi(i)+1,\psi(\varphi(i+\delta))-1)=\tilde{\psi}^*(i+\delta-1).$$
Similarly, if $\nu(i-1,\tilde{\psi}^*(i+\delta-1)-1)>0$ and $\tilde{\psi}^*(i+\delta-1)\neq i+\delta$, then there must exist an odd number $c_{\psi(i-1)}$ such that $\nu(\psi(i-1),\tilde{\psi}^*(i+\delta-1)-1)<0$. Now let $i'=\psi(i-1)$, and similarly we have
$$\tilde{\psi}^*(i+\delta-1)=\tilde{\psi}^*(\psi(i-1),\psi(i+\delta-1)-1)=\tilde{\psi}^*(i',i'+\delta-1)=\tilde{\psi}(i',i'+\delta-1).$$ The first equation $\tilde{\psi}^*(i+\delta-1)=\tilde{\psi}^*(\psi(i-1),\psi(i+\delta-1)-1)$ is obtained from Lemma \ref{pr2}, since otherwise it would hold that  $\tilde{\psi}^*(\psi(i-1),i+\delta-1)< i$, in which case\\ $\nu(i-1,\tilde{\psi}^*(i+\delta-1)-1)<0$ contrary to the supposition. Alternatively, if $\tilde{\psi}^*(i+\delta-1)= i+\delta$, then $\mu(1,2h)>0$. Hence we have covered all the cases, completing the proof of the lemma.
\end{proof}

The proof of the above lemma provides a particularly good example. We note that when $\nu(i,\tilde{\psi}(j)-1)<0$, $i<\tilde{\psi}(j)-1$ must exist. According to Lemma \ref{le6}, we know that $i<\varphi(i)<\tilde{\psi}(j)-1$ must exist. Below, we show that for every $i,j$ that make $\nu(i,\tilde{\psi}(j)-1)<0$, we can uniquely find $i',j'$ that make $\nu(i'-1,\tilde{\psi}^*(i',j')-1)>0$ so that the polynomial addition of the two is strictly positive.

Returning to our example, we take a simple case. We observe one of the negative polynomials generated by the negative point $c_1$. $\varphi(1)=3$, but a positive point is generated by the next point past the negative point $c_2$ to which a positive polynomial corresponds. If the product term of the negative polynomial does not exceed its own length limit, we follow the direction of the second type of iteration across negative points in the same way as across positive points. For example,
$$b_{c_1}b_2b_{c_2}\rightarrow b_4b_5b_{c_3}b_{c_4}\leftarrow b_8 \leftarrow b_9\leftarrow b_{10}.$$
Here we are using the second type of iteration, our crossing of negative points may cause the length of the product term at both ends of the route will exceed $2h$. Here is one example.
$$b_{c_1}b_2b_{c_2}b_4b_5b_{c_3},$$
$$b_4b_5b_{c_3}\leftarrow b_{c_4}b_8b_9b_{10}b_{c_1}\leftarrow b_{2}.$$
At this point, we must intercept the appropriate route into two or three positive polynomials.
$$b_{c_1}b_2b_{c_2}\rightarrow b_4b_5b_{c_3} \leftarrow b_{c_4}b_8b_9b_{10}b_{c_1},$$
$$\mu(1,10),$$
$$-\mu(1,3)\mu(4,6)\mu(7,11).$$
The length of the product term of the third polynomial above exceeds $2h$, and the length of the product term of the second polynomial is $2h$. It is obvious that the magnitude of the second term is greater than that of the third term, so the sum of these two polynomials is strictly positive.

For the general case, by the proof of Lemma \ref{le7}, if there exist $i,j$ satisfying events $A_i$, $C^j_i$ and $\nu(i,\tilde{\psi}(j)-1)<0$, then we can uniquely find $i'=\varphi(i)+1$ that satisfies the event $A_{i'-1}$ and $\nu(i'-1,\tilde{\psi}(j)-1)>0$. Then $\tilde{\psi}(j)-1\ge i'-1$, and if our iterative route of $j$ does not include $\tilde{\psi}(j)=i$, then $\tilde{\psi}(j)> i'$, $\tilde{\psi}(i',j)\neq \varphi(i')+1$. We then have $\tilde{\psi}(j)=\tilde{\psi}^*(i',j)$.
$$\nu(i,i'-1)\tilde{\nu}(i'-1,\tilde{\psi}(j)-1)\leftarrow\cdots \leftarrow b_{c_{j+1}-1},$$
$$\tilde{\nu}(i'-1,\tilde{\psi}^*(i',j')-1)\leftarrow\cdots \leftarrow b_{c_{j+1}-1}\leftarrow\cdots \leftarrow b_{c_{j'+1}-1}.$$
Because the starting points of the two product terms are different, the end point $c_{j+1}-1$ corresponding to the original starting point does not necessarily satisfy the event $C^j_{i'-1}$. But there must be a unique $j'$ that satisfies both event $C^{j'}_{i'-1}$ and $\tilde{\psi}^*(i',j)=\tilde{\psi}^*(i',j')$. By Lemma \ref{le6} and Lemma \ref{pr2}, the new iterative route must contain $c_{j+1}-1$ at this time, and it will be connected to the original route to continue the iteration.\\
If $j=j'$, then we add the two polynomials to get a positive polynomial as follows
$$b_{c_i}\cdots b_{c_{i'-1}}\rightarrow\tilde{\nu}(i'-1,\tilde{\psi}^*(i',j')-1)\leftarrow\cdots \leftarrow b_{c_{j+1}-1}.$$
If $j<j'$, then there must exist
\begin{align}\label{eq39}
i+\delta \le \varphi(j+1)<j'+1,
\end{align}
and we have $\varphi(j'+1)\ge i'-1+\delta=\varphi(i)+\delta$ or $\varphi(j'+1)$ does not exist. Then $v(j'+1,\varphi(i)+\delta)>0$, otherwise by Lemma \ref{le6}, there must exist $\varphi(j'+1)<\varphi(i)+\delta$ or $\psi(\varphi(i)+\delta)=i+\delta>j'+1$, that is contrary to Inequality (\ref{eq39}). And by Lemma \ref{pr2}, we summarize the sign of the function $v(m,n)$ between the negative point $c_{j+1}$ and the negative point $c_{\varphi(i)+\delta}$ as follows.
$$\underset{j+1}{\mid} \cdots \ -\ \cdots \overset{i+\delta}{\mid}\ \  \cdots  \ + \ \cdots      \underset{\varphi(j+1)}{\mid}\cdots  \ - \ \cdots  \ \     \underset{j'+1}{\mid}\ \ \cdots \ +\ \cdots  \overset{\varphi(i+\delta)}{\mid}.$$
Then we have $\nu(\varphi(j+1),j'+1)<0$. \\
If $\tilde{\nu}(i'-1,\varphi(j+1))>0$, then we will directly divide the polynomial route into two routes, each with a positive polynomial route.
$$b_{c_i}\cdots b_{c_{i'-1}}\rightarrow\tilde{\nu}(i'-1,\tilde{\psi}^*(i',j')-1)\leftarrow\cdots \leftarrow b_{c_{j+1}-1},$$
$$\tilde{\nu}(i'-1,\varphi(j+1))\leftarrow b_{c_{\varphi(j+1)}+1}\leftarrow\cdots \leftarrow b_{c_{j'+1}-1}.$$
Furthermore, if $\tilde{\nu}(i'-1,\varphi(j+1))<0$, then by the inequality $\nu(\varphi(j+1),j'+1)<0$, we know that $\mu(c_{i'-1}+1,c_{j'+1}-1)>0$. Then, we can reintegrate a positive polynomial route, and the sum of the other two terms is also positive.
$$b_{c_i}\cdots b_{c_{i'-1}}\rightarrow\tilde{\nu}(i'-1,\tilde{\psi}^*(i',j')-1)\leftarrow\cdots \leftarrow b_{c_{j+1}-1}\leftarrow b_{c_{j+1}}\cdots b_{c_{\varphi(j+1)}},$$
$$\mu(c_{i'-1}+1,c_{j'+1}-1),$$
$$-\mu(c_{i'-1}+1,c_{j+1}-1)\mu(c_{j+1},c_{\varphi(j+1)}).$$
Indeed, $\lvert\mu(c_{i'-1}+1,c_{j'+1}-1)\rvert\le \lvert\mu(1,2h)\rvert \le \lvert\mu(c_{i'-1}+1,c_{j+1}-1)\mu(c_{j+1},c_{\varphi(j+1)})\rvert$, so the sum of the last two items above is positive.

That is to say, when $c_i$ is an odd number, then for a negative polynomial satisfying the indicative function on the right-hand side of the inequality of $Q$ in Lemma \ref{le7}, we can always find a unique positive polynomial corresponding to it, and the sum of the two is always positive. Furthermore, following Lemma \ref{le6}, each negative polynomial corresponds to a positive polynomial, which also uniquely corresponds to the original negative polynomial. Similarly, when $c_i$ is even, the same logic applies, the only difference being that $\tilde{\psi}^*(j)=\varphi(i)+1$ must be considered additionally. Here $i'=\varphi(i)+1$, $\tilde{\psi}(i ',j')=i'=\tilde{\psi}^*(j)$, and the same conclusion can be obtained. So we have $Q>0$, which proves that Theorem \ref{th3} completes our conjecture verification.

\section*{Acknowledgments}
We have benefited from comments by the referee. We are very grateful to him for pointing out some errors in our paper and providing detailed revision suggestions. In particular, he gave an example to help readers better understand the proof of the conjecture. Chen and Liang gratefully acknowledge the support of the National Key R\&D Program of China (grant No. ZR2019ZD41).

\section*{Declarations}

\subsection*{Conflict of interest:} The authors declared that they have no conflicts of interest to this work.

\subsection*{Data availability statement:} All relevant data are within the paper.
\bibliography{sn-bibliography}
\bibliographystyle{unsrt}

\end{document}